\documentclass[a4paper]{article}%
\usepackage{graphicx}
\usepackage[authoryear,round,semicolon]{natbib}
\usepackage{float}
\usepackage{color}
\usepackage{amsmath}
\usepackage{amsfonts}
\usepackage{amssymb}
\usepackage{hyperref}%
\providecommand{\U}[1]{\protect\rule{.1in}{.1in}}
\newtheorem{theorem}{Theorem}
\newtheorem{acknowledgement}[theorem]{Acknowledgement}
\newtheorem{algorithm}[theorem]{Algorithm}

\newtheorem{corollary}[theorem]{Corollary}

\newtheorem{example}[theorem]{Example}

\newtheorem{proposition}[theorem]{Proposition}
\newtheorem{remark}[theorem]{Remark}

\newenvironment{proof}[1][Proof]{\noindent\textbf{#1.} }{\ \rule{0.5em}{0.5em}}
\graphicspath {{c:/Sophie/DocSwp/Recherche/HaiHa/Ancien/Etude3/MCAP/version2/}}
\begin{document}

\title{A random shock model with mixed effect, including competing soft and sudden
failures, and dependence}
\author{Sophie MERCIER\thanks{Laboratoire de Math\'{e}matiques et de leurs
Applications -- Pau (UMR CNRS 5142), Universit\'{e} de Pau et des Pays de
l'Adour, B\^{a}timent IPRA, Avenue de l'Universit\'{e}, F-64013 PAU cedex,
FRANCE; sophie.mercier@univ-pau.fr} \ \& Hai Ha PHAM\thanks{Research and
Development Institute, Duytan University, K7/25 Quang Trung, Da Nang, VIETNAM;
phamhaiha09@gmail.com} }
\date{Preprint accepted for publication in \emph{Methodology and Computing in
Applied Probability}, 31 August 2014,\\
\url{http://dx.doi.org/10.1007/s11009-014-9423-6}}
\maketitle

\begin{abstract}
A system is considered, which is subject to external and possibly fatal
shocks, with dependence between the fatality of a shock and the system age.
Apart from these shocks, the system suffers from competing soft and sudden
failures, where soft failures refer to the reaching of a given threshold for
the degradation level, and sudden failures to accidental failures,
characterized by a failure rate. A non-fatal shock increases both degradation
level and failure rate of a random amount, with possible dependence between
the two increments. The system reliability is calculated by four different
methods. Conditions under which the system lifetime is New Better than Used
are proposed. The influence of various parameters of the shocks environment on
the system lifetime is studied.



\end{abstract}

\section{Introduction}

This paper is devoted to the survival analysis of a system subject to
competing failure modes within an external stressing environment. The external
environment is assumed to stress the system at random and isolated times
according to a random shock model. Such a model can represent external demands
e.g., which put some stress on the system at their arrivals. Shock models have
been the subject of an extensive literature. Following
\citet{mallor2003classification}, shock models may be classified into
different categories, according to whether arrival times and shock magnitudes
are correlated \citep{mallor2001shocks, mallor2003reliability}, or independent
\citep{cha2009terminating}, and according to the assumption put on the shocks
arrival process: homogeneous or non-homogeneous Poisson process
\citep{abdul1972nonstationary, cha2009terminating, cha2007study,
cha2011stochastic, esary1973shock, qian1999cumulative}, renewal process
\citep{skoulakis2000} or non-stationary pure birth process
\citep{proschan1975shock}. According to the influence of shocks on the system,
shock models may be further classified into three types: extreme shock models
(a shock can cause the system immediate failure, see
\citealt{gut1999extreme}), cumulative shock models (a shock increases some
intrinsic characteristic of the system such as its deterioration, failure
rate, age, number of already endured shocks, ..., see \citealt{cha2007study,
qian1999cumulative}) and mixed shock models (a shock can either cause the
system immediate failure or increase some intrinsic characteristic, see
\citealt{cha2011stochastic, gut2001mixed}). More references can be found in
\color{black}%
\citet{finkelstein2013stochastic}%
\color{black}%
; \citet{mallor2003classification};
\color{black}%
\citet{nakagawa2007shock}%
\color{black}%
; \citet{singpurwalla1995survival}.

In this paper, a shock model with mixed effects is considered, where the
occurrence of shocks is classically modelled through a non-homogeneous Poisson
process and where each shock may result in the system immediate failure
through a Bernoulli trial, independent of the system intrinsic behaviour.
\color{black}%
Apart from shocks, the system suffers from competing soft and sudden failures,
where soft failures refer to the reaching of some given threshold for the
degradation level, and sudden failures to accidental failures, characterized
by a failure rate. A possible representation corresponds to a two-component
series system, where the first component is subject to soft failures and the
second one to sudden failure. The system may hence fail through three
different competing modes: traumatic failure due to a fatal shock; soft
failure; sudden failure. Each non fatal shock induces some increase of both
deterioration level and failure, with possible dependence in-between.

In the oldest literature, most models considered one single possible type of
failures for the system: e.g. soft failures in \cite{marshall1979multivariate}
or traumatic failures due to shocks in \cite{savits1988some}, both in the
multivariate setting. More recently, different models have been developed,
which consider two different types of failures. For instance, competition
between soft and sudden failures is studied in \cite{Zhu2010}. Several
application cases are proposed in the paper (see also references therein). An
industrial example is also provided in \cite{Wang2014}, which studies the
reliability of an aircraft
\color{black}%
engine. Note that \cite{Zhu2010} assumes soft and sudden failures to be
independent whereas the stressing environment of the present paper makes them
dependent. Competition between soft failures and occurrence of a traumatic
event are considered in Degradation-Threshold-Shock models (DTS-models,
denomination of \citealt{lehmann2006}) which have been proposed by
\cite{lemoine1985} and further studied by \cite{lehmann2006, lehmann2009}. A
case study is provided in \cite{Hao2013} for the analysis of fatigue crack
growth, where the effects of shocks on the degradation are put into evidence.
Note that, contrary to the present paper, DTS-models consider some possible
influence of the deterioration level of the system on the shocks arrival rate.
However, the first shock of a DTS-model is always fatal to the system (leading
to a single shock possibly endured by the system), whereas successive non
fatal shocks are here envisioned. We could not find any paper which takes into
account three competing failure modes, as in the present paper. However, based
on the case studies from the previous literature, one can think that our model
can reflect lots of systems subject to competing soft and hard failures,
within a stressing environment. For instance, one can think of an aerial cabin
hooked to a cable: the supporting cable is deteriorating due to corrosion and
fatigue; the linking pulley is subject to sudden failures (and maybe also to
deterioration); both cable and pulley endure shocks at each cabin travel,
which increase jointly their respective deterioration level and failure rate.

The present paper consider several kinds of dependence between the three
competing failure modes: at each non fatal shock, the increase of
deterioration and failure rate is simultaneous. This induces a first type of
dependence between soft and sudden failures. Each shock may induce a failure,
either because the shock is fatal, or because the deterioration is suddenly
increased beyond the threshold level. This induces a second type of dependence
between soft and traumatic failures, which may be simultaneous. Also, some
possible dependence is envisioned between the increments of failure rate and
of deterioration at each non fatal shock. This induces a third type of
dependence between soft and sudden failures.
\color{black}%
Finally, following \cite{cha2009terminating, cha2011stochastic}, the
probability for a shock to be fatal depends on the shock arrival time, which
induces a last type of dependence. Up to our knowledge, all these kinds of
dependence have not been yet considered altogether and, as will be seen all
along the text, this model enlarges several ones from the previous literature.

The paper is organized as follows: the model is specified in Section
\ref{model}. The system reliability is computed through different methods in
Section 3.
\color{black}%
Sufficient conditions are provided in Section 4\textbf{\ }for the system
lifetime to be New Better than Used. The influence of various parameters of
the shock environment on the lifetime is studied in Section 5.
\color{black}%
Numerical experiments are proposed in Section \ref{section Num} and concluding
remarks end the paper in Section \ref{concl}.

\section{The model\label{model}}%

\color{black}%
To make the model clear, a two-component series system is considered, where
the first component is subject to sudden failure and the second one to soft
failure. This two-unit system is just a representation for the competing soft
and sudden failure modes, with no restriction.
\color{black}%
In the ideal condition (for example in a laboratory environment), the lifetime
of the first component is characterized by its intrinsic hazard rate $h(t)$,
$t\geq0$ while the second one is subject to some accumulative deterioration
modeled by an increasing stochastic process $(G_{t})_{t\geq0}$ (e.g. a gamma
process). The second component fails once its deterioration level exceeds a
failure threshold $L$. The lifetimes of the two components are made dependent
by their common stressing environment. This environment is modelled by a
random shock process, where the shocks arrive according to a non-homogeneous
Poisson process $(N_{t})_{t\geq0}$ with intensity $d\Lambda(t)=\lambda(t)dt$
(or cumulated intensity $\Lambda(t)$). To avoid useless technical details, we
assume that $\Lambda\left(  t\right)  >0$ for all $t>0$. More generally, one
might consider that $\Lambda\left(  t\right)  >0$ only for $t$ greater than
some $t_{0}>0$, which would mean that the shocks would arrive only after time
$t_{0}$. The points of the Poisson process are denoted by $T_{1}$, ...,
$T_{n}$, ... with $T_{0}=0<T_{1}<\ldots<T_{n}<\ldots$ almost surely. A shock
at time $t$ may cause the system immediate failure (fatal shock) with
probability $p(t)\in\lbrack0,1]$, which depends on the age $t$ of the system
at the shock arrival. A shock at time $t$ is non fatal with probability
$q\left(  t\right)  =1-p\left(  t\right)  $. A non fatal shock at time $T_{i}$
increases the deterioration of both components in a different way:

\begin{itemize}
\item for the first component, its hazard rate is increased of a non negative
random amount $V_{i}^{(1)}$,

\item for the second one, its accumulated deterioration is increased of a non
negative random amount $V_{i}^{(2)}$.
\end{itemize}

The random vectors $V_{i}=\left(  V_{i}^{(1)},V_{i}^{(2)}\right)  $, $i=1$,
$2$, ... are assumed to be independent and identically distributed (i.i.d.)
with common distribution $\mu(dv_{1},dv_{2})$, and independent of the shocks
arrival times $(T_{n})_{n\geq1}$ (and hence independent of the Poisson process
$(N_{t})_{t\geq0}$). At each shock, the increments $V_{i}^{(1)}$ and
$V_{i}^{(2)}$ are possibly dependent. When subscript $i$ is unnecessary, we
drop it and set $V=\left(  V^{\left(  1\right)  },V^{\left(  2\right)
}\right)  $ to be a generic copy of $V_{i}=\left(  V_{i}^{(1)},V_{i}%
^{(2)}\right)  $. For $j=1,2$, the distribution of $V^{\left(  j\right)  }$ is
denoted by $\mu_{j}(dv_{j})$.

We set
\color{black}%
$\left(  A_{t}\right)  _{t\geq0}$ to be the bivariate compound Poisson process
defined by
\begin{equation}
A_{t}=\left(  \sum_{k=1}^{N_{t}}V_{k}^{(1)},\sum_{k=1}^{N_{t}}V_{k}%
^{(2)}\right)  =\left(  A_{t}^{\left(  1\right)  },A_{t}^{\left(  2\right)
}\right) \label{At_i}%
\end{equation}
with%
\begin{align*}
A_{t}^{(1)}  & ={\sum_{k=1}^{N_{t}}V_{k}^{(1)},}\\
A_{t}^{(2)}  & ={\sum_{k=1}^{N_{t}}V_{k}^{(2)},}%
\end{align*}%
\color{black}%
where $\sum_{k=1}^{0}\cdots=0$.

The processes $(A_{t})_{t\geq0}$ and $(G_{t})_{t\geq0}$ are assumed to be independent.

Provided that the system is functioning up to time $t$, the random variables
$A_{t}^{(1)}$ and $A_{t}^{(2)}$ stand for the cumulated increments on $\left[
0,t\right]  $ of the failure rate of the first component and of the
deterioration of the second component due to the external environment,
respectively. Setting $\mathcal{F}_{t}=\sigma\left(  A_{s},s\leq t\right)  $
to be the $\sigma-$field generated by $\left(  A_{s}\right)  _{s\leq t}$ and
provided that the system is still up at time $t$, the conditional hazard rate
of the first component given $\mathcal{F}_{t}$ is%
\[
X_{t}^{(1)}=h(t)+A_{t}^{(1)}%
\]
and the conditional deterioration of the second component given $\mathcal{F}%
_{t}$ is%
\[
X_{t}^{(2)}=G_{t}+A_{t}^{(2)}.
\]

To make it clearer, we introduce $\tau_{i}$, $i=1,2$ to be the lifetime of the
$i^{th}$ component under the external environment, without taking into account
the possibility of fatal shocks for the system. To simplify the writing, we
denote by $\mathbb{P}\left(  B|\mathcal{F}_{t}\right)  $ the conditional
expectation $\mathbb{E}\left(  \mathbf{1}_{B}|\mathcal{F}_{t}\right)  $ for
any measurable set $B$, where $\mathbf{1}_{B}$ stands for the indicator
function ($\mathbf{1}_{B}\left(  \omega\right)  =1$ if $\omega\in B$, 0
elsewhere). We then have:%
\begin{align}
\mathbb{P}\left(  \tau_{1}>t|\mathcal{F}_{t}\right)   &  =e^{-\int_{0}%
^{t}X_{s}^{(1)}~ds}=e^{-H\left(  t\right)  }e^{-\int_{0}^{t}A_{s}^{(1)}%
~ds},\label{tau1}\\
\mathbb{P}\left(  \tau_{2}>t|\mathcal{F}_{t}\right)   &  =\mathbb{P}\left(
X_{t}^{(2)}\leq L|\mathcal{F}_{t}\right)  =\mathbb{P}\left(  G_{t}\leq
L-A_{t}^{(2)}|\mathcal{F}_{t}\right)  =F_{G_{t}}\left(  L-A_{t}^{\left(
2\right)  }\right) \label{tau2}%
\end{align}
where
\begin{equation}
H\left(  t\right)  =\int_{0}^{t}h(s)~ds\label{H}%
\end{equation}
is the cumulated intrinsic failure rate of the first component and where
$F_{G_{t}}$ stands for the cumulative distribution function (c.d.f.) of $G_{t}
$. (Recall that $G_{t}$ is independent of $\mathcal{F}_{t}$).

We now let $\tau_{3}$ to be the time to the first fatal shock for the system
with
\[
\tau_{3}=\inf\left(  n\geq1:\text{ the shock at time }T_{n}\text{ is
fatal}\right)
\]
and we assume that the Bernoulli trials (fatal shocks or not) which happen at
each shock arrival are independent one with each other, and that they depend
on $\mathcal{F}_{t}$ only through the $q\left(  T_{n}\right)  $'s, that is:
\begin{equation}
\mathbb{P}\left(  \tau_{3}>t|\mathcal{F}_{t}\right)  =\prod_{i=1}^{\infty
}q\left(  T_{i}\right)  \mathbf{1}_{\left[  0,t\right]  }\left(  T_{i}\right)
=\prod_{i=1}^{N_{t}}q\left(  T_{i}\right) \label{tau3}%
\end{equation}
where $\prod_{i=1}^{0}\cdots=1$.

The system failure is induced either by a fatal shock or by a component
failure
\color{black}%
(soft or sudden failure)%
\color{black}%
, whatever arrives first. The lifetime of the system hence is
\begin{equation}
\tau=\min\left(  \tau_{1},\tau_{2},\tau_{3}\right)  .\label{tau}%
\end{equation}

We finally make the additional assumption that $\tau_{1}$, $\tau_{2}$ and
$\tau_{3}$ are conditionally independent given $\mathcal{F}_{t}:$%
\begin{align*}
&  \mathbb{P}\left(  \tau_{1}>t,\tau_{2}>t,\tau_{3}>t|\mathcal{F}_{t}\right)
\\
&  =\mathbb{P}\left(  \tau_{1}>t|\mathcal{F}_{t}\right)  \mathbb{P}\left(
\tau_{2}>t|\mathcal{F}_{t}\right)  \mathbb{P}\left(  \tau_{3}>t|\mathcal{F}%
_{t}\right)  .
\end{align*}

To sum up, the whole model is specified by:

\begin{itemize}
\item $\left(  V^{\left(  1\right)  },V^{\left(  2\right)  }\right)  :$ the
(generic) random increments in failure rate (first component, $V^{\left(
1\right)  }$) and deterioration (second component, $V^{\left(  2\right)  }$),

\item $\lambda\left(  x\right)  dx:$ the intensity of the non-homogeneous
Poisson process,

\item $h\left(  x\right)  :$ the intrinsic failure rate of the first
component
\color{black}%
(sudden failure)%
\color{black}%
,

\item $(G_{t})_{t\geq0}:$ the intrinsic deterioration of the second component
\color{black}%
(soft failure)%
\color{black}%
,

\item $p\left(  t\right)  :$ the probability for a shock at time $t$ to be
fatal at the system level,
\end{itemize}

(plus some independence assumptions).

\medskip

By taking special cases for these five ingredients, we can see that our model
extents some well-known models from the literature.

For instance, taking $p\left(  t\right)  =0$ (no fatal shocks), $V^{\left(
1\right)  }=$ constant, $V^{\left(  2\right)  }=0$ and $G_{t}=0$ all $t\geq0$
(no second component), one gets the "stochastic failure model in random
environment" from \cite{cha2007study}.

Taking $V^{\left(  2\right)  }=0$ and $G_{t}=0$ all $t\geq0$ (one single
component), one gets the "stochastic survival model for a system under
randomly variable environment" from \cite{cha2011stochastic}.

Taking $G_{t}=0$ all $t\geq0$, $V^{\left(  1\right)  }=V^{\left(  2\right)
}=0$, the model resumes to a classical extreme shock model (one single
component), where system failures are only due to shocks arriving according to
a non-homogeneous Poisson process, with probability $p\left(  t\right)  $ for
a shock to be fatal (and $q\left(  t\right)  $ to be harmless). This model is
interpreted as the Brown-Proschan model by \cite{cha2009terminating}; see also
\cite{brown1983imperfect} where various properties of the model are explored.

Taking $p\left(  t\right)  =0$ (no fatal shocks), $V^{\left(  1\right)
}=h\left(  t\right)  =0$ (one single component), $\lambda\left(  x\right)
=\lambda$ (homogeneous Poisson process), $G_{t}=0$ all $t\geq0$ (no intrinsic
deterioration for the second - and single - component), one gets the
"cumulative damage threshold model" from \cite[Section 4, case of i.i.d.
damage increments]{esary1973shock}.

Taking $p\left(  t\right)  =p$ (constant), $V^{\left(  1\right)  }=h\left(
t\right)  =0$ (one single component), $G_{t}=0$ all $t\geq0$, one gets the
"cumulative damage model with two kinds of shocks" from \cite[case of i.i.d.
damage increments]{qian1999cumulative}.

Taking $V^{\left(  1\right)  }=h=0$, $V^{\left(  2\right)  }$ exponentially
distributed, $G_{t}=t$, all $t\geq0$, one gets the model from Subsection
\textbf{3.b} in \cite{cha2009terminating}.

All these models are summed up in Table \ref{table1}. Note that we do not
pretend at any exhaustibility and our model will include lots of other
previous models which are not provided here.

\begin{table}[ptb]
\caption{A few particular models from the literature}%
\label{table1}
\begin{center}%
\begin{tabular}
[c]{cc}\hline
Brown-Proschan model from \citep{brown1983imperfect} & $h=V^{\left(  1\right)
}=V^{\left(  2\right)  }=0$,\\
& $G_{t}=0$, $\forall t\geq0$\\\hline
Deterministic boundary in \citep{cha2009terminating} & $V^{\left(  1\right)
}=h=0$, $G_{t}=t$\\
& $V^{\left(  2\right)  }$ exponentially distributed\\\hline
\cite{cha2007study} & $q=1$, $V^{(1)}$ is a constant,\\
& $V^{\left(  2\right)  }=0$, $G_{t}=0$, $\forall t\geq0$\\\hline
\cite{cha2011stochastic} & $V^{\left(  2\right)  }=0$, $G_{t}=0$, $\forall
t\geq0$\\\hline
Cumulative damage threshold models & $q=1$, $V^{\left(  2\right)  }=0$,
$\lambda$ is a constant\\
in Section 4 of \citep{esary1973shock} & $h=0$, $G_{t}=0$, $\forall t\geq
0$\\\hline
\cite{qian1999cumulative} & $q$ is a constant, $h=V^{\left(  1\right)  }=0$,\\
& $G_{t}=0$, $\forall t\geq0$\\\hline
\end{tabular}
\vspace*{0.5cm}
\end{center}
\end{table}

\section{%
\color{black}%
Calculation of the system reliability\label{rel}}%

\color{black}%
The objective of this section is to calculate the reliability $R_{L}(t)$ of
the system at time $t$%
\color{black}%
, with%
\[
R_{L}(t)=\mathbb{P}(\tau>t),\text{ all }t\geq0,
\]
where we recall that the system lifetime $\tau$ is defined by $\left(
\ref{tau}\right)  $.

A first way to compute $R_{L}(t)$ is to use classical Monte-Carlo simulations
and to simulate a large number of independent histories for the system up to
time $t$ (\textbf{Method 1}). This method will serve as a comparison tool in
the numerical experimentations in Section \ref{section Num}.
\color{black}%
This method requires the simulation of a random variable with conditional
hazard rate $h(t)+{\sum_{k=1}^{N_{t}}V_{k}^{(1)}}$ (see Algorithm \ref{algo}
in Section \ref{section Num} for details) and may imply long computational
times for the system reliability. We provide below a few alternate methods
which may be quicker and also easier to implement.%
\color{black}%

\begin{proposition}
[Method 2]\label{prop1}The reliability is given by
\begin{equation}
R_{L}(t)=\mathbb{P}(\tau>t)=e^{-H(t)}\phi_{t}(L),\label{reliability}%
\end{equation}
where $H\left(  t\right)  $ is provided by $\left(  \ref{H}\right)  $ and
where
\begin{align}
\phi_{t}(L) &  =\mathbb{E}\left(  F_{G_{t}}\left(  L-A_{t}^{\left(  2\right)
}\right)  e^{-\int_{0}^{t}A_{s}^{(1)}~ds}\prod_{i=1}^{N_{t}}q\left(
T_{i}\right)  \right) \label{eq_phi1}\\
&  =\mathbb{E}\left(  F_{G_{t}}\left(  L-\sum_{i=1}^{N_{t}}V_{i}^{(2)}\right)
e^{-\sum_{i=1}^{N_{t}}V_{i}^{(1)}(t-T_{i})}\prod_{i=1}^{N_{t}}q(T_{i})\right)
\label{eq_phi2}%
\end{align}

\end{proposition}

\begin{proof}
Due to the conditional independence of $\tau_{1}$, $\tau_{2}$ and $\tau_{3}$
given $\mathcal{F}_{t}$, we have:
\begin{align*}
R_{L}(t)  &  =\mathbb{P}(\tau_{1}>t,\tau_{2}>t,\tau_{3}>t)\\
&  =\mathbb{E}\left(  \mathbb{P}\left(  \tau_{1}>t,\tau_{2}>t,\tau
_{3}>t|\mathcal{F}_{t}\right)  \right) \\
&  =\mathbb{E}\left(  \mathbb{P}\left(  \tau_{1}>t|\mathcal{F}_{t}\right)
\mathbb{P}\left(  \tau_{2}>t|\mathcal{F}_{t}\right)  \mathbb{P}\left(
\tau_{3}>t|\mathcal{F}_{t}\right)  \right)  .
\end{align*}
Using $\left(  \ref{tau1},\ref{tau2},\ref{tau3}\right)  $, we get:%
\[
R_{L}(t)=\mathbb{E}\left(  e^{-H\left(  t\right)  }e^{-\int_{0}^{t}A_{s}%
^{(1)}~ds}~F_{G_{t}}\left(  L-A_{t}^{\left(  2\right)  }\right)  ~\prod
_{i=1}^{N_{t}}q\left(  T_{i}\right)  \right)  ,
\]
which provides $\left(  \ref{eq_phi1}\right)  $ and next $\left(
\ref{eq_phi2}\right)  $, due to $\left(  \ref{At_i}\right)  $ and
\begin{align}
\int_{0}^{t}A_{s}^{(1)}~ds  &  =\int_{0}^{t}\sum_{i=1}^{+\infty}V_{i}%
^{(1)}\mathbf{1}_{\left\{  T_{i}\leq s\right\}  }~ds=\sum_{i=1}^{+\infty}%
V_{i}^{(1)}\int_{0}^{t}\mathbf{1}_{\left\{  T_{i}\leq s\right\}
}~ds\nonumber\\
&  =\sum_{i=1}^{+\infty}V_{i}^{(1)}(t-T_{i})\mathbf{1}_{\left\{  T_{i}\leq
t\right\}  }=\sum_{i=1}^{N_{t}}V_{i}^{(1)}(t-T_{i}).\label{int A1}%
\end{align}

\end{proof}%

\color{black}%
Based on the previous result, the only point to get the reliability $R_{L}(t)
$ is to compute $\phi_{t}(L)$. The remaining of the section is hence devoted
to the computation of $\phi_{t}(L)$. Starting from $\left(  \ref{eq_phi1}%
\right)  $ (or $\left(  \ref{eq_phi2}\right)  $), a possibility is to compute
$\phi_{t}(L)$\ through Monte-Carlo simulations of $\left(  N_{t}\right)
_{t\geq0}$ and $\left(  A_{t}\right)  _{t\geq0}$, which is simpler and quicker
than simulating trajectories of the system according to the initial model.
\color{black}%
This method is called \textbf{Method 2} in Section \ref{section Num}.

Following \cite{abdul1972nonstationary} and \cite{esary1973shock}, one may
also use a series expansion of $\phi_{t}(L)$, as provided by the following proposition.

\begin{proposition}
[Method 3, general case]\label{prop2}
\begin{equation}
\phi_{t}(L)=e^{-\Lambda(t)}\sum_{n=0}^{\infty}P_{n}\left(  t,L\right)
\frac{(\Lambda(t))^{n}}{n!}\label{eq_phi_3}%
\end{equation}
where%
\begin{equation}
P_{n}\left(  t,L\right)  =\mathbb{E}\left(  F_{G_{t}}\left(  L-\sum_{i=1}%
^{n}V_{i}^{(2)}\right)  \prod_{i=1}^{n}\left(  q(Z_{i})e^{-(t-Z_{i}%
)V_{i}^{(1)}}\right)  \right) \label{PnGen}%
\end{equation}
and $(Z_{i})_{i\geq1}$ are i.i.d. random variables with probability density
function (p.d.f.) $\frac{\lambda(x)}{\Lambda(t)}1_{[0,t]}(x)$ and independent
of $\left(  V_{i}\right)  _{i\geq1}$.
\end{proposition}

\begin{proof}
Conditioning on the Poisson process $\left(  N_{t}\right)  _{t\geq0}$, we
have:%
\[
\phi_{t}(L)=\mathbb{E}\left(  f\left(  N_{t}\right)  \right)  =\sum
_{n=0}^{+\infty}f\left(  n\right)  \frac{(\Lambda(t))^{n}}{n!}e^{-\Lambda(t)}%
\]
with%
\[
f\left(  n\right)  =\mathbb{E}\left(  F_{G_{t}}\left(  L-\sum_{i=1}^{N_{t}%
}V_{i}^{(2)}\right)  \prod_{i=1}^{N_{t}}q(T_{i})e^{-\sum_{i=1}^{N_{t}}%
(t-T_{i})V_{i}^{(1)}}|N_{t}=n\right)  .
\]
Now, given that $N_{t}=n$, the conditional joint distribution of
$(T_{1},\ldots T_{n})$ is the same as the joint distribution of the order
statistics $(Z_{\left(  1\right)  },\ldots,Z_{\left(  n\right)  })$ of $n$
i.i.d. random variables $Z_{1},\ldots,Z_{n}$ with p.d.f. $\frac{\lambda
(x)}{\Lambda(t)}1_{[0,t]}(x)$ (see \citealt{cocozza1998processus} e.g.). Using
the fact that $V_{i}=\left(  V_{i}^{\left(  1\right)  },V_{i}^{\left(
2\right)  }\right)  $ is independent of $\left(  N_{t}\right)  _{t\geq0}$, we
get:%
\[
f\left(  n\right)  =\mathbb{E}\left(  F_{G_{t}}\left(  L-\sum_{i=1}^{n}%
V_{i}^{(2)}\right)  \prod_{i=1}^{n}q(Z_{\left(  i\right)  })e^{-\sum_{i=1}%
^{n}(t-Z_{\left(  i\right)  })V_{i}^{(1)}}\right)  .
\]
Noting that the expression within the expectation is invariant through
permutation of the $Z_{i}$'s, we derive that:%
\[
f\left(  n\right)  =\mathbb{E}\left(  F_{G_{t}}\left(  L-\sum_{i=1}^{n}%
V_{i}^{(2)}\right)  \prod_{i=1}^{n}q(Z_{i})e^{-\sum_{i=1}^{n}(t-Z_{i}%
)V_{i}^{(1)}}\right)  ,
\]
which provides the result.
\end{proof}

\begin{remark}
Based on the previous result, one can see that our model is equivalent to a
classical shock model, where the shocks arrive according to a non-homogeneous
Poisson process with intensity $d\Lambda(x)$ and with conditional probability
of survival at time $t$ equal to $P_{n}\left(  t,L\right)  $, given that there
has been $n$ shocks up to time $t$.
\end{remark}

\begin{corollary}
[Method 3, independent case]\label{cor1}In the special case where $V^{\left(
1\right)  }$ and $V^{(2)}$ are independent, we get:%
\begin{equation}
\phi_{t}(L)=e^{-\Lambda(t)}\sum_{n=0}^{\infty}Q_{n}\left(  t,L\right)
\frac{(a(t))^{n}}{n!}\label{eq1}%
\end{equation}
with
\begin{equation}
Q_{n}\left(  t,L\right)  =\mathbb{E}\left(  F_{G_{t}}\left(  L-\sum_{i=1}%
^{n}V_{i}^{(2)}\right)  \right) \label{Pn indep}%
\end{equation}
and {%
\begin{equation}
a(t)=\left(  \tilde{\mu}_{1}\ast(q\lambda)\right)  (t)=%
\color{black}%
\int_{0}^{t}\tilde{\mu}_{1}(z)(q\lambda)(t-z)~dz,\label{def a(t)}%
\end{equation}
}%
\color{black}%
{where $\ast$ stands for the convolution operator, $(q\lambda)(x)=q(x)\lambda
(x)$, all }$x\geq0$%
\color{black}%
, and $\tilde{\mu}_{1}$ stands for the Laplace transform of the distribution
$\mu_{1}$ of $V^{\left(  1\right)  }$, with%
\[
\tilde{\mu}_{1}\left(  s\right)  =\int_{0}^{+\infty}e^{-xs}\mu_{1}\left(
dx\right)  \text{, all }s\geq0.
\]

\end{corollary}%

\color{black}%

\begin{proof}
Starting from $\left(  \ref{PnGen}\right)  $ and using the independence of all
$V_{i}^{\left(  1\right)  }$'s, $V_{i}^{\left(  2\right)  }$'s and $Z_{i}$'s,
and the identical distributions of all $V_{i}^{\left(  1\right)  }$'s and of
all $Z_{i}$'s, we get:%
\begin{equation}
P_{n}\left(  t,L\right)  =\mathbb{E}\left(  F_{G_{t}}\left(  L-\sum_{i=1}%
^{n}V_{i}^{(2)}\right)  \right)  \left(  \mathbb{E}\left(  q(Z_{1}%
)e^{-(t-Z_{1})V^{(1)}}\right)  \right)  ^{n}\label{Pn2}%
\end{equation}
with%
\begin{align*}
\mathbb{E}\left(  q(Z_{1})e^{-(t-Z_{1})V^{(1)}}\right)   &  =\frac{1}%
{\Lambda(t)}\int_{0}^{t}\lambda(z)q\left(  z\right)  \mathbb{E}\left(
e^{-(t-z)V^{(1)}}\right)  dz\\
&  =\frac{1}{\Lambda(t)}\int_{0}^{t}\left(  \lambda q\right)  (z)\tilde{\mu
}_{1}\left(  t-z\right)  dz=\frac{a\left(  t\right)  }{\Lambda\left(
t\right)  }.
\end{align*}
Substituting this expression into Eq. $\left(  \ref{Pn2}\right)  $ and next
into Eq. $\left(  \ref{PnGen}\right)  $ provides the result.
\end{proof}

\begin{example}
\label{ex1}Let $V^{(2)}$ be an exponentially distributed with mean $1/\theta$
and $V^{\left(  1\right)  }$ an independent random variable. In that case,
$\sum_{i=1}^{n}{V}_{i}^{(2)}$ is Gamma distributed with parameter $\left(
n,\theta\right)  $. This provides%
\[
Q_{n}\left(  t,L\right)  =\int_{0}^{L}F_{G_{t}}(L-x)\frac{\theta^{n}x^{n-1}%
}{(n-1)!}e^{-\theta x}~dx,\,\text{all }n\geq1.
\]
Using $Q_{0}\left(  t,L\right)  =F_{G_{t}}\left(  L\right)  $ and Eq. $\left(
\ref{eq1}\right)  $, we get:%
\[
R_{L}(t)=e^{-H(t)-\Lambda(t)}\left(  F_{G_{t}}(L)+\sum_{n=1}^{\infty}%
\frac{(a(t))^{n}}{n!}\int_{0}^{L}F_{G_{t}}(L-x)\frac{\theta^{n}x^{n-1}%
}{(n-1)!}e^{-\theta x}~dx\right)  .
\]
Taking $G_{t}=t$ for all $t\geq0$, $V^{\left(  1\right)  }=0$ and $h\left(
t\right)  =0$ as a special case, we get%
\begin{align*}
F_{G_{t}}\left(  L\right)   &  =\mathbf{1}_{\left\{  t\leq L\right\}  },\\
a(t) &  =\int_{0}^{t}\left(  q\lambda\right)  \left(  s\right)  ds
\end{align*}
and, for $t\leq L:$%
\begin{align*}
R_{L}(t) &  =e^{-\Lambda(t)}\left(  1+\sum_{n=1}^{\infty}\frac{(\int_{0}%
^{t}\left(  q\lambda\right)  \left(  s\right)  ds)^{n}}{n!}\int_{0}^{L-t}%
\frac{\theta^{n}x^{n-1}}{(n-1)!}e^{-\theta x}~dx\right) \\
&  =e^{-\Lambda(t)}\left(  1+\sum_{n=1}^{\infty}\frac{(\int_{0}^{t}\left(
q\lambda\right)  \left(  s\right)  ds)^{n}}{n!}\sum_{k=n}^{+\infty}%
\frac{\left(  \theta\left(  L-t\right)  \right)  ^{k}}{k!}e^{-\theta\left(
L-t\right)  }\right)
\end{align*}
using successive integrations by parts for the last integral. This last
expression is the result of Theorem 2 in \cite{cha2009terminating}, which
hence appears as a special case of the previous results.
\end{example}

From Proposition \ref{prop2} and Corollary \ref{cor1}, one may derive the
following approximation for $\phi_{t}\left(  L\right)  $.

\begin{corollary}
[Method 3, approximation]\label{Cor Method3}For $N\geq0$, let $\phi_{t}%
^{N}(L)$ be defined by
\[
\phi_{t}^{N}(L)=e^{-\Lambda(t)}\sum_{n=0}^{N}P_{n}\left(  t,L\right)
\frac{(\Lambda(t))^{n}}{n!}%
\]
in the general case, and by
\[
\phi_{t}^{N}(L)=e^{-\Lambda(t)}\sum_{n=0}^{N}Q_{n}\left(  t,L\right)
\frac{(a(t))^{n}}{n!}%
\]
in case $V^{\left(  1\right)  }$ and $V^{(2)}$ are independent, where
$P_{n}\left(  t,L\right)  $ and $Q_{n}\left(  t,L\right)  $ are provided by
$\left(  \ref{PnGen}\right)  $ and $\left(  \ref{Pn indep}\right)  $,
respectively. Then, for all $t\geq0$, the sequence $\left(  \phi_{t}%
^{N}(L)\right)  _{N\geq1}$ increases to the limit $\phi_{t}(L)$ when
$N\rightarrow\infty$ and for all $N\geq0$, we have
\[
\phi_{t}^{N}(L)\leq\phi_{t}(L)\leq\phi_{t}^{N}(L)+\epsilon_{N}(t)
\]
where%
\[
\epsilon_{N}(t)=e^{-\Lambda(t)}\sum_{n=N+1}^{\infty}\frac{(a\left(  t\right)
)^{n}}{n!}=e^{-\left(  \Lambda(t)-a\left(  t\right)  \right)  }\mathbb{P}%
(Y_{t}>N),
\]
$a\left(  t\right)  $ is defined by $\left(  \ref{def a(t)}\right)  $ and
$Y_{t}$ is Poisson distributed with mean $a\left(  t\right)  $.
\end{corollary}

\begin{proof}
We just look at the general case. From Proposition \ref{prop2}, we have
\[
\phi_{t}(L)-\phi_{t}^{N}(L)=e^{-\Lambda(t)}\sum_{n=N+1}^{\infty}P_{n}\left(
t,L\right)  \frac{(\Lambda(t))^{n}}{n!}.
\]
Due to $F_{G_{t}}\left(  L-\sum_{i=1}^{n}V_{i}^{(2)}\right)  \leq1$, we get:%
\[
P_{n}\left(  t,L\right)  \leq\mathbb{E}\left(  \prod_{i=1}^{n}\left(
q(Z_{i})e^{-(t-Z_{i})V_{i}^{(1)}}\right)  \right)  =\left(  \frac{a\left(
t\right)  }{\Lambda\left(  t\right)  }\right)  ^{n}%
\]
based on the proof of Corollary \ref{cor1}. This provides:%
\[
0\leq\phi_{t}(L)-\phi_{t}^{N}(L)\leq e^{-\Lambda(t)}\sum_{n=N+1}^{\infty}%
\frac{(a\left(  t\right)  )^{n}}{n!}=e^{-\left(  \Lambda(t)-a\left(  t\right)
\right)  }\sum_{n=N+1}^{\infty}e^{-a\left(  t\right)  }\frac{(a\left(
t\right)  )^{n}}{n!}%
\]
and the result.
\end{proof}

The previous proposition provides numerical bounds for $\phi_{t}(L)$, which
may be adjusted as tight as necessary, taking $N$ large enough.
\color{black}%
Also, the required number of terms in the truncated series is given, to get a
specified precision. This method is quite adapted as soon as it is possible to
compute the $P_{n}\left(  t,L\right)  $'s (or the $Q_{n}\left(  t,L\right)
$'s). This mostly requires the distribution of ${\sum_{i=1}^{n}V_{i}^{(2)}}$
to be known in full form, which is the case e.g. when the $V_{i}^{\left(
2\right)  }$'s are constant or Gamma distributed. An example is provided in
Example \ref{ex1}, in the special case of an exponential distribution. In the
most general case, the computation of the $P_{n}\left(  t,L\right)  $'s (or of
the $Q_{n}\left(  t,L\right)  $'s) may be as difficult as the initial problem
of computing $\phi_{t}\left(  L\right)  $, so that the previous method is not
always adapted.

We finally present another method based on Laplace transform, which does not
suffer from the same restriction.%
\color{black}%

\begin{theorem}
[Method 4]\label{prop3}We have:%
\begin{equation}
\tilde{\phi}_{t}(s)=\tilde{F}_{G_{t}}(s)\tilde{\nu}_{t}(s)\text{, all }%
s\geq0\label{eq_phi3}%
\end{equation}
or equivalently
\begin{equation}
\phi_{t}(L)=\left(  F_{G_{t}}\ast\nu_{t}\right)  (L)\text{, all }%
L\geq0,\label{eq_phi_4}%
\end{equation}
where $\nu_{t}$ is provided by its Laplace transform
\begin{equation}
\tilde{\nu}_{t}(s)=e^{-\Lambda(t)+\left(  (q\lambda)\ast\tilde{\mu}%
(\cdot,s)\right)  (t)},\label{nu}%
\end{equation}
with $\tilde{\mu}$ the bivariate Laplace transform of the distribution $\mu$
of $\left(  V^{\left(  1\right)  },V^{\left(  2\right)  }\right)  :$%
\[
\tilde{\mu}(u,s)=\iint_{\mathbb{R}_{+}^{2}}e^{-uv_{1}-sv_{2}}\mu(dv_{1}%
,dv_{2})\text{, all }u,s\geq0,
\]
and $\tilde{\mu}(\cdot,s):u\rightarrow\tilde{\mu}(u,s)$, all $s\geq0$.\newline
In the special case where $V^{\left(  1\right)  }$ and $V^{(2)}$ are
independent, $\tilde{\nu}_{t}(s)$ may be simplified into:%
\[
\tilde{\nu}_{t}(s)=e^{-\Lambda(t)+a(t)\tilde{\mu}_{2}(s)}%
\]
where $a\left(  t\right)  $ is provided by $\left(  \ref{def a(t)}\right)  $
and where $\tilde{\mu}_{2}(s)$ is the univariate Laplace transform of $\mu
_{2}$.
\end{theorem}

\begin{proof}
Remembering that $A_{t}^{(2)}=\sum_{i=1}^{N_{t}}V_{i}^{(2)}$, we get from
$\left(  \ref{eq_phi2}\right)  $ that:%
\begin{align}
\tilde{\phi}_{t}(s) &  =\int_{0}^{\infty}e^{-sL}~\mathbb{E}\left[  F_{G_{t}%
}\left(  L-A_{t}^{(2)}\right)  e^{-\sum_{i=1}^{N_{t}}V_{i}^{(1)}(t-T_{i}%
)}\prod_{i=1}^{N_{t}}q(T_{i})\right]  ~dL\nonumber\\
&  =\mathbb{E}\left[  \left(  \int_{0}^{\infty}e^{-sL}F_{G_{t}}\left(
L-A_{t}^{(2)}\right)  ~dL\right)  e^{-\sum_{i=1}^{N_{t}}V_{i}^{(1)}(t-T_{i}%
)}\prod_{i=1}^{N_{t}}q(T_{i})\right] \label{lap fi}%
\end{align}
with%
\[
\int_{0}^{\infty}e^{-sL}F_{G_{t}}\left(  L-A_{t}^{(2)}\right)  ~dL=\int%
_{A_{t}^{(2)}}^{\infty}e^{-sL}F_{G_{t}}\left(  L-A_{t}^{(2)}\right)  ~dL
\]
because $G_{t}$ is non negative. Setting $w=L-A_{t}^{(2)}$, we obtain
\[
\int_{A_{t}^{(2)}}^{\infty}e^{-sL}F_{G_{t}}\left(  L-A_{t}^{(2)}\right)
~dL=e^{-sA_{t}^{(2)}}\int_{0}^{\infty}e^{-sw}F_{G_{t}}(w)~dw=e^{-sA_{t}^{(2)}%
}\tilde{F}_{G_{t}}(s).
\]
Substituting this expression into $\left(  \ref{lap fi}\right)  $ provides:%
\begin{equation}
\tilde{\phi}_{t}(s)=\tilde{F}_{G_{t}}(s)~\theta\left(  s\right)
\label{lap phi}%
\end{equation}
with%
\begin{align*}
\theta\left(  s\right)   &  =\mathbb{E}\left[  e^{-s\sum_{i=1}^{N_{t}}%
V_{i}^{(2)}}e^{-\sum_{i=1}^{N_{t}}V_{i}^{(1)}(t-T_{i})}\prod_{i=1}^{N_{t}%
}q(T_{i})\right] \\
&  =\mathbb{E}\left[  e^{-\sum_{i=1}^{N_{t}}\left(  V_{i}^{(1)}(t-T_{i}%
)+sV_{i}^{(2)}-\ln q(T_{i})\right)  }\right]  .
\end{align*}
Because $i\leq N_{t}$ is equivalent to $T_{i}\leq t$, we get:%
\[
\theta\left(  s\right)  =\mathbb{E}\left(  e^{-\sum_{i=1}^{\infty}\psi
_{s,t}(V_{i}^{(1)},V_{i}^{(2)},T_{i})}\right)
\]
with%
\begin{equation}
\psi_{s,t}(v_{1},v_{2},w)=\left(  (t-w)v_{1}+sv_{2}-\ln q(w)\right)
\mathbf{1}_{\{w\leq t\}}.\label{psi}%
\end{equation}
Noting that the sequence $\left(  V_{n}^{\left(  1\right)  },V_{n}^{\left(
2\right)  },T_{n}\right)  _{n\geq0}$ are the points of a Poisson random
measure $M$ with intensity $\nu\left(  dv_{1},dv_{2},dw\right)  =\mu
(dv_{1},dv_{2})\lambda(w)dw$, the function $\theta\left(  s\right)  $ may be
interpreted as a Laplace functional with respect of $M:$
\[
\theta\left(  s\right)  =\mathbb{E}\left(  e^{-M\psi_{s,t}}\right)  .
\]

The formula for Laplace functionals of Poisson random measures \cite[Theorem
2.9]{ccnlar2011probability} next provides:%
\begin{equation}
\theta\left(  s\right)  =\exp\left(  -\iiint_{\mathbb{R}_{+}^{3}}\left(
1-e^{-\psi_{s,t}}\right)  d\nu\right) \label{theta}%
\end{equation}
with%
\[
\iiint_{\mathbb{R}_{+}^{3}}\left(  1-e^{-\psi_{s,t}}\right)  d\nu
=\iiint_{\mathbb{R}_{+}^{3}}\left(  1-e^{-\psi_{s,t}(v_{1},v_{2},w)}\right)
\mu(dv_{1},dv_{2})\lambda(w)dw.
\]
Substituting $\psi_{s,t}$ by its expression $\left(  \ref{psi}\right)  $, we
get:%
\begin{align*}
&  \iiint_{\mathbb{R}_{+}^{3}}\left(  1-e^{-\psi_{s,t}}\right)  d\nu\\
&  =\int_{0}^{t}\left(  \iint_{\mathbb{R}_{+}^{2}}\left(  1-e^{-\left(
(t-w)v_{1}+sv_{2}-\ln q(w)\right)  )}\right)  \mu(dv_{1},dv_{2})\right)
\lambda(w)~dw\\
&  =\int_{0}^{t}\left(  1-q\left(  w\right)  \iint_{\mathbb{R}_{+}^{2}%
}e^{-\left(  (t-w)v_{1}+sv_{2}\right)  )}\mu(dv_{1},dv_{2})\right)
\lambda(w)~dw\\
&  =\Lambda\left(  t\right)  -\int_{0}^{t}\left(  q\lambda\right)  \left(
w\right)  \tilde{\mu}(t-w,s)~dw\\
&  =\Lambda(t)-\left[  (q\lambda)\ast(\tilde{\mu}(\cdot,s))\right]  (t).
\end{align*}
Substituting this expression into $\left(  \ref{theta}\right)  $ and next into
$\left(  \ref{lap phi}\right)  $ provides
\[
\tilde{\phi}_{t}(s)=\tilde{F}_{G_{t}}(s)\tilde{\nu}_{t}(s),
\]
with $\tilde{\nu}_{t}(s)$ given by $\left(  \ref{nu}\right)  $. Equation
$\left(  \ref{eq_phi_4}\right)  $ is a direct consequence.

Finally, in case $V^{\left(  1\right)  }$ and $V_{1}^{(2)}$ are independent,
we have:%
\[
\tilde{\mu}(w,s)=\tilde{\mu}_{1}(w)\tilde{\mu}_{2}(s)
\]
and%
\begin{equation}
\tilde{\nu}_{t}(s)=e^{-\Lambda(t)+\left(  (q\lambda)\ast\tilde{\mu}%
_{1}\right)  (t)~\tilde{\mu}_{2}(s)}=e^{-\Lambda(t)+a(t)\tilde{\mu}_{2}%
(s)},\label{eq3}%
\end{equation}
which ends this proof.
\end{proof}

Based on the previous result, one can compute $\phi_{t}(L)$ by inverting its
Laplace transform with respect of $L$. Looking at Equation $\left(
\ref{eq_phi_4}\right)  $, the key point is the inversion of the Laplace
transform $\tilde{\nu}_{t}(s)$.
\color{black}%
We next provide an example where the inversion is possible in full form. In
the most general case,
\color{black}%
this can be done numerically using some Laplace inversion software.

\begin{example}
Let $h=0$, $G_{t}=0$ (all $t\geq0$), $\lambda$ and $q$ constant, $V^{\left(
1\right)  }=V^{\left(  2\right)  }$ identically exponentially distributed with
mean $1/\theta$ (so that $V^{\left(  1\right)  }$ and $V^{\left(  2\right)  }$
are completely dependent). Then:%
\[
\tilde{\mu}(u,s)=\iint_{\mathbb{R}_{+}^{2}}e^{-uv_{1}-sv_{2}}\theta e^{-\theta
v_{1}}~dv_{1}~\delta_{v_{1}}\left(  dv_{2}\right)  =\frac{\theta}{u+s+\theta},
\]
where $\delta_{v_{1}}$ stands for the Dirac mass at $v_{1}$. We easily get:%
\[
\left(  (q\lambda)\ast\tilde{\mu}(\cdot,s)\right)  (t)=q\lambda\int_{0}%
^{t}\frac{\theta}{u+s+\theta}du=q\lambda\theta~\ln\left(  \frac{t+s+\theta
}{s+\theta}\right)
\]
and%
\[
\tilde{\nu}_{t}(s)=e^{-\lambda t}\left(  1+\frac{t}{s+\theta}\right)
^{q\lambda\theta}.
\]
For $t<\theta$, we have
\[
\tilde{\nu}_{t}(s)=e^{-\lambda t}\left(  1+\sum_{n=0}^{\infty}\binom
{q\lambda\theta}{n+1}\left(  \frac{t}{s+\theta}\right)  ^{n+1}\right)
\]
where
\[
\binom{q\lambda\theta}{n}=\frac{q\lambda\theta(q\lambda\theta-1)\ldots
(q\lambda\theta-n+1)}{n!}.
\]
Inverting the Laplace transform $\tilde{\nu}_{t}(s)$, we obtain:
\[
\nu_{t}(dx)=e^{-\lambda t}\left(  \delta_{0}(dx)+\sum_{n=0}^{\infty}%
\binom{q\lambda\theta}{n+1}\frac{t^{n+1}}{n!}e^{-\theta x}x^{n}~dx\right)  .
\]
As $F_{G_{t}}=1$, we get the following full form for the reliability:%
\begin{align*}
R_{L}(t) &  =(1\ast\nu_{t})(L)\\
&  =e^{-\lambda t}\left(  1+\sum_{n=0}^{\infty}\binom{q\lambda\theta}%
{n+1}\left(  \frac{t}{\theta}\right)  ^{n+1}F_{n+1,\theta}(L)\right)  ,
\end{align*}
where $F_{n+1,\theta}$ is the cumulative distribution function of a gamma
distributed random variable with parameter $(n+1,\theta)$.
\end{example}

In the special case where $V^{\left(  1\right)  }$ and $V^{(2)}$ are
independent, the Laplace inversion of $\tilde{\nu}_{t}(s)$ is reduced to
inverting $e^{a(t)\tilde{\mu}_{2}(s)}$, or equivalently to inverting
$e^{C~\tilde{\mu}_{2}(s)}$, where $C$ is a constant.
\color{black}%
This is hence easier than in the most general case of correlated $V^{\left(
1\right)  }$ and $V^{(2)}$.

\medskip

To sum up the section, we have at our disposal four different methods for
computing the reliability:

\begin{description}
\item[\textbf{Method 1} (Direct MC simulations)]
\color{black}%
The main drawbacks of this method are that it suffers from long computation
times and that its implementation is less direct than for the other methods.

\item[\textbf{Method 2} (Computing $\phi_{t}\left(  L\right)  $ through
formula $\left(  \ref{eq_phi2}\right)  $ and MC simulations)] This method is
much quicker and much easier to implement than Method 1. Besides, it is always
possible to use it. However, Method 3 (when possible) and Method 4 are quicker.

\item[\textbf{Method 3} (Truncated series expansion + control of the
truncation error through Corollary \ref{Cor Method3})] This method provides
very good results as soon as the $P_{n}\left(  t,L\right)  $'s (or the
$Q_{n}\left(  t,L\right)  $'s) are available in full form.

\item[\textbf{Method 4} (Laplace transform inversion)] This method is the best
when it is possible to inverse the Laplace transform $\tilde{\nu}_{t}(s) $ in
full form. Numerical Laplace inversion also provides quite good results.
\end{description}%

\color{black}%

\section{%
\color{black}%
An ageing property for the system lifetime}

Let us recall that a random variable $Z$ (or $\bar{F}_{Z}$) is New Better than
Used (NBU) if
\begin{equation}
\mathbb{P}(Z>s+t)\leq\mathbb{P}(Z>s)\mathbb{P}(Z>t),\label{NBU}%
\end{equation}
all $s,t\geq0$. We here provide sufficient conditions under which $\tau$ is NBU.

\begin{theorem}
\label{propNBU}%
\color{black}%
Assume that the intrinsic lifetimes of both components are NBU, which means
that:
\begin{align}
e^{-H\left(  s+t\right)  }  & \leq e^{-H\left(  s\right)  }e^{-H\left(
t\right)  }\text{, all }s,t\geq0,\label{NBU1}\\
F_{G_{t+s}}\left(  l\right)   & \leq F_{G_{t}}\left(  l\right)  F_{G_{s}%
}\left(  l\right)  \text{, all }l,s,t\geq0,\label{NBU2}%
\end{align}
where the second condition is true as soon as $(G_{t})_{t\geq0}$ is a
univariate non negative L\'{e}vy process.%
\color{black}
\newline Then, $\tau$ is NBU if one among the two following conditions is satisfied:

\begin{enumerate}
\item $q$ is non increasing and $\lambda$ is constant,

\item $q$ is constant and $\Lambda$ is super-additive ($\Lambda(x+y)\geq
\Lambda(x)+\Lambda(y)$, all $x,y\geq0$).
\end{enumerate}
\end{theorem}

\begin{proof}%
\color{black}%
Let us first note that, in case $(G_{t})_{t\geq0}$ is a univariate non
negative L\'{e}vy process, we have:%
\begin{align*}
F_{G_{t+s}}\left(  u\right)   &  =\mathbb{P}\left(  G_{t}+\left(
G_{t+s}-G_{t}\right)  \leq u\right) \\
&  \leq\mathbb{P}\left(  G_{t}\leq u,G_{t+s}-G_{t}\leq u\right) \\
&  =\mathbb{P}\left(  G_{t}\leq u\right)  \mathbb{P}\left(  G_{t+s}-G_{t}\leq
u\right) \\
&  =F_{G_{t}}\left(  u\right)  F_{G_{s}}\left(  u\right)
\end{align*}
due to the independent and homogenous increments of $(G_{t})_{t\geq0}$ for the
third line. Assumption $\left(  \ref{NBU2}\right)  $ is hence true.%
\color{black}%

Starting again from
\[
\mathbb{P}(\tau>t)=e^{-H\left(  t\right)  }\phi_{t}(L),
\]
and based on the NBU assumption $\left(  \ref{NBU1}\right)  $, it is
sufficient to show that $\phi_{t}(L)$ is NBU. Now:
\begin{align*}
& \phi_{s+t}(L)\\
& =\mathbb{E}\left(  F_{G_{t+s}}\left(  L-\sum_{i=1}^{N_{t+s}}V_{i}%
^{(2)}\right)  e^{-\sum_{i=1}^{N_{t+s}}V_{i}^{(1)}(t+s-T_{i})}\prod
_{i=1}^{N_{t+s}}q(T_{i})\right) \\
& \leq\mathbb{E}\left(  F_{G_{t}}\left(  L-\sum_{i=1}^{N_{t+s}}V_{i}%
^{(2)}\right)  \times F_{G_{s}}\left(  L-\sum_{i=1}^{N_{t+s}}V_{i}%
^{(2)}\right)  e^{-\sum_{i=1}^{N_{t+s}}V_{i}^{(1)}(t+s-T_{i})}\prod
_{i=1}^{N_{t+s}}q(T_{i})\right)
\end{align*}

due to the second NBU assumption $\left(  \ref{NBU2}\right)  $.

Under each of the two provided conditions, $q$ is non increasing so that
$q(T_{i})\leq q(T_{i}-t)$, all $i\geq N_{t}+1$. Using
\begin{align*}
-\sum_{i=1}^{N_{t}}V_{i}^{(1)}(t+s-T_{i}) &  \leq-\sum_{i=1}^{N_{t}}%
V_{i}^{(1)}(t-T_{i}),\\
L-\sum_{i=1}^{N_{t+s}}V_{i}^{(2)} &  \leq L-\sum_{i=1}^{N_{t}}V_{i}^{(2)},\\
L-\sum_{i=1}^{N_{t+s}}V_{i}^{(2)} &  \leq L-\sum_{i=N_{t}+1}^{N_{t+s}}%
V_{i}^{(2)},
\end{align*}
and splitting the exponential and the product into two parts, one gets:%
\begin{align}
\phi_{s+t}(L) &  \leq\mathbb{E}\left[  F_{G_{t}}\left(  L-\sum_{i=1}^{N_{t}%
}V_{i}^{(2)}\right)  e^{-\sum_{i=1}^{N_{t}}V_{i}^{(1)}(t-T_{i})}\prod
_{i=1}^{N_{t}}q(T_{i})\right. \nonumber\\
&  \times\left.  F_{G_{s}}\left(  L-\sum_{i=N_{t}+1}^{N_{t+s}}V_{i}%
^{(2)}\right)  e^{-\sum_{i=N_{t}+1}^{N_{t+s}}V_{i}^{(1)}(t+s-T_{i})}%
\prod_{i=N_{t}+1}^{N_{t+s}}q(T_{i}-t)\right]  .\label{Edfi}%
\end{align}
Setting
\[
T_{i}^{(t)}=T_{N_{t}+i}-t\,,\,\text{all }i\geq1,
\]
then $\left(  T_{n}^{(t)}\right)  _{n\geq1}$ are points of the Poisson process
$\left(  N_{s}^{\left(  t\right)  }=N_{t+s}-N_{t}\right)  _{s\geq0}$ with
admits $\lambda(t+x)~dx$ for intensity. Equation $\left(  \ref{Edfi}\right)  $
now writes:%
\begin{align*}
\phi_{s+t}(L) &  \leq\mathbb{E}\left[  F_{G_{t}}\left(  L-\sum_{i=1}^{N_{t}%
}V_{i}^{(2)}\right)  e^{-\sum_{i=1}^{N_{t}}V_{i}^{(1)}(t-T_{i})}\prod
_{i=1}^{N_{t}}q(T_{i})\right. \\
&  \times\left.  F_{G_{s}}\left(  L-\sum_{j=1}^{N_{s}^{\left(  t\right)  }%
}V_{j+N_{t}}^{(2)}\right)  e^{-\sum_{j=1}^{N_{s}^{\left(  t\right)  }%
}V_{j+N_{t}}^{(1)}(s-T_{j}^{\left(  t\right)  })}\prod_{j=1}^{N_{s}^{\left(
t\right)  }}q(T_{j}^{\left(  t\right)  })\right]  ,
\end{align*}
or equivalently:%
\begin{align*}
\phi_{s+t}(L) &  \leq\sum_{n=0}^{+\infty}\mathbb{E}\left[  \mathbf{1}%
_{\left\{  N_{t}=n\right\}  }F_{G_{t}}\left(  L-\sum_{i=1}^{N_{t}}V_{i}%
^{(2)}\right)  e^{-\sum_{i=1}^{N_{t}}V_{i}^{(1)}(t-T_{i})}\prod_{i=1}^{N_{t}%
}q(T_{i})\right. \\
&  \times\left.  F_{G_{s}}\left(  L-\sum_{j=1}^{N_{s}^{\left(  t\right)  }%
}V_{j+n}^{(2)}\right)  e^{-\sum_{j=1}^{N_{s}^{\left(  t\right)  }}%
V_{j+n}^{(1)}(s-T_{j}^{\left(  t\right)  })}\prod_{j=1}^{N_{s}^{\left(
t\right)  }}q(T_{j}^{\left(  t\right)  })\right]  .
\end{align*}
As $\left(  N_{s}^{\left(  t\right)  }\right)  _{s\geq0}$ is independent on
$\left(  N_{u}\right)  _{u\leq t}$ and as the $V_{i}$'s are i.i.d. and
independent on $\left(  N_{u}\right)  _{u\leq t}$, one gets:%
\[
\phi_{s+t}(L)\leq\sum_{n=1}^{+\infty}a_{n}b_{n}%
\]
with%
\begin{align*}
a_{n} &  =\mathbb{E}\left[  \mathbf{1}_{\left\{  N_{t}=n\right\}  }F_{G_{t}%
}\left(  L-\sum_{i=1}^{N_{t}}V_{i}^{(2)}\right)  e^{-\sum_{i=1}^{N_{t}}%
V_{i}^{(1)}(t-T_{i})}\prod_{i=1}^{N_{t}}q(T_{i})\right] \\
b_{n} &  =\mathbb{E}\left[  F_{G_{s}}\left(  L-\sum_{j=1}^{N_{s}^{\left(
t\right)  }}V_{j+n}^{(2)}\right)  e^{-\sum_{j=1}^{N_{s}^{\left(  t\right)  }%
}V_{j+n}^{(1)}(s-T_{j}^{\left(  t\right)  })}\prod_{j=1}^{N_{s}^{\left(
t\right)  }}q(T_{j}^{\left(  t\right)  })\right]  \text{.}%
\end{align*}
Noting that $b_{n}$ is independent on $n$ ($b_{n}=b_{0}$, all $n\geq0$) and
that $\sum_{n=1}^{+\infty}a_{n}=\phi_{t}(L)$, we finally have%
\[
\phi_{s+t}(L)\leq\phi_{t}(L)\times\phi_{s}^{\left(  t\right)  }\left(
L\right)
\]
where
\[
\phi_{s}^{\left(  t\right)  }\left(  L\right)  =b_{0}=\mathbb{E}\left[
F_{G_{s}}\left(  L-\sum_{j=1}^{N_{s}^{\left(  t\right)  }}V_{j}^{(2)}\right)
e^{-\sum_{j=1}^{N_{s}^{\left(  t\right)  }}V_{j}^{(1)}(s-T_{j}^{\left(
t\right)  })}\prod_{j=1}^{N_{s}^{\left(  t\right)  }}q(T_{j}^{\left(
t\right)  })\right]  .
\]
The point now is to prove that $\phi_{s}^{\left(  t\right)  }\left(  L\right)
\leq\phi_{s}\left(  L\right)  $ under the two different assumptions.

\begin{enumerate}
\item If $\lambda$ is a constant, then $\left(  N_{s}^{\left(  t\right)
}\right)  _{s\geq0}$ is identically distributed as $(N_{u})_{u\geq0}$. We
hence have%
\[
\phi_{s}^{\left(  t\right)  }\left(  L\right)  =\phi_{s}\left(  L\right)
\]
and the result is clear.

\item If $q$ is constant, then:%
\[
\phi_{s}^{\left(  t\right)  }\left(  L\right)  =\mathbb{E}\left[  f_{\infty
}\left(  \left(  T_{i}^{(t)}\right)  _{i=1}^{\infty}\right)  \right]  ,
\]
where
\[
f_{n}\left(  \left(  t_{i}\right)  _{i=1}^{n}\right)  =\mathbb{E}\left(
F_{G_{s}}\left(  L-\sum_{j=1}^{n}V_{j}^{(2)}\mathbf{1}_{\{t_{j}\leq
s\}}\right)  e^{-\sum_{j=1}^{n}V_{j}^{(1)}(s-t_{j})\mathbf{1}_{\{t_{j}\leq
s\}}}q^{\sum_{j=1}^{n}\mathbf{1}_{\{t_{j}\leq s\}}}\right)
\]
for $n\in\mathbb{N}^{\ast}\mathbb{\cup}\left\{  \infty\right\}  $. Moreover,
the respective cumulated intensities of $(N_{u})_{u\geq0}$ and $\left(
N_{u}^{(t)}\right)  _{u\geq0}$ are $\Lambda(x)$ and $\Lambda(x+t)-\Lambda(t)$,
with%
\[
\Lambda(x+t)-\Lambda(t)\geq\Lambda(x)\text{, all }t,x\geq0
\]
due to the super-additivity of $\Lambda(x)$. We derive from \cite[Theorem
6.B.40, Example 6.B.41]{shaked2006stochastic} that
\[
\left(  T_{i}\right)  _{i=1}^{n}\geq_{sto}\left(  T_{i}^{(t)}\right)
_{i=1}^{n}%
\]
for all $n\geq1$, where $\geq_{sto}$ stand for the standard stochastic order.
As $f_{n}$ is non decreasing with respect to each $t_{i}$, we get that:%
\[
\mathbb{E}\left[  f_{n}\left(  \left(  T_{i}\right)  _{i=1}^{n}\right)
\right]  \geq\mathbb{E}\left[  f_{n}\left(  \left(  T_{i}^{(t)}\right)
_{i=1}^{n}\right)  \right]
\]
for each $n\in\mathbb{N}^{\ast}$. Setting $n\rightarrow+\infty$, we derive by
Lebesgue's dominated convergence theorem that%
\[
\lim\limits_{n\rightarrow+\infty}\mathbb{E}\left[  f_{n}\left(  \left(
T_{i}\right)  _{i=1}^{n}\right)  \right]  =\phi_{s}\left(  L\right)  \geq
\lim\limits_{n\rightarrow+\infty}\mathbb{E}\left[  f_{n}\left(  \left(
T_{i}^{(t)}\right)  _{i=1}^{n}\right)  \right]  =\phi_{s}^{\left(  t\right)
}\left(  L\right)  ,
\]
which achieves the proof.
\end{enumerate}
\end{proof}%

\color{black}%
The conditions of the previous theorem means that

\begin{enumerate}
\item the probability for a shock to be non fatal decreases with time (and
$\lambda$ is constant),

\item the cumulated rate of shocks arrivals is larger at time $t$ than at time
$t=0$ (and $q$ is constant).
\end{enumerate}

Such conditions hence mean that the environment is more and more stressing, or
that it is more stressing after a while than at the beginning. Such conditions
are quite natural.

\section{%
\color{black}%
Influence of the dependence induced by the stressing environment on the system
lifetime}%

\color{black}%
We here study the influence on the lifetime $\tau$ of different parameters of
the stressing environment: we study the influence of probability $q\left(
\cdot\right)  $, of the dependence between $V^{\left(  1\right)  }$ and
$V^{(2)}$ and of the cumulated intensity function $\Lambda$. We hence look at
the influence on the lifetime $\tau$ of all characteristics of the stressing
environment which make the components dependent.
\color{black}%
The influence of $q\left(  \cdot\right)  $ is straightforward. We mention it
for sake of completeness.

\subsection{Influence of $q\left(  \cdot\right)  $ on the lifetime $\tau$}

Let us consider two different systems with identical parameters except from
$q\left(  \cdot\right)  $ (first system) and $\tilde{q}\left(  \cdot\right)  $
(second system) and such that $q(w)\leq\tilde{q}(w)$ for all $w\geq0$. Then,
adding a tilde ($\sim$) to any quantity referring to the second system, we
directly get from $\left(  \ref{eq_phi1}\right)  $ that
\[
R_{L}(t)\leq\tilde{R}_{L}(t),\ \text{all }t\geq0,
\]
or equivalently that $\tau$ is smaller than $\tilde{\tau}$ in the sense of the
standard stochastic order $(\tau\leq_{st}\tilde{\tau}):$%
\begin{equation}
\mathbb{P}\left(  \tau>t\right)  \leq\mathbb{P}\left(  \tilde{\tau}>t\right)
\text{, all }t\geq0.\label{tau sto}%
\end{equation}
As expected, the lifetime $\tau$ is hence increasing with the probability
$q\left(  \cdot\right)  $ for a shock to be non fatal.

\subsection{Influence of the dependence between $V^{\left(  1\right)  }$ and
$V^{(2)}$ on the lifetime $\tau$}

We here study the influence of the dependence between the two marginal
increments $V^{\left(  1\right)  }$ and $V^{(2)}$ on the lifetime $\tau$. To
measure the dependence level between $V^{\left(  1\right)  }$ and $V^{(2)}$,
we use the lower (or upper) orthant order, where we recall that $V=\left(
V^{\left(  1\right)  },V^{(2)}\right)  $ is said to be smaller than $\tilde
{V}=\left(  \tilde{V}^{\left(  1\right)  },\tilde{V}^{(2)}\right)  $ in the
lower orthant order ($V\leq_{lo}\tilde{V}$) if
\begin{equation}
\mathbb{P}\left(  V^{\left(  1\right)  }\leq x_{1},V^{(2)}\leq x_{2}\right)
\leq\mathbb{P}\left(  \tilde{V}^{\left(  1\right)  }\leq x_{1},\tilde{V}%
^{(2)}\leq x_{2}\right)  \text{, all }x_{1},x_{2}\in\mathbb{R},\label{lo}%
\end{equation}
or equivalently if
\begin{equation}
\mathbb{P}\left(  V^{\left(  1\right)  }>x_{1},V^{(2)}>x_{2}\right)
\leq\mathbb{P}\left(  \tilde{V}^{\left(  1\right)  }>x_{1},\tilde{V}%
^{(2)}>x_{2}\right)  \text{, all }x_{1},x_{2}\in\mathbb{R}.\label{uo}%
\end{equation}

\begin{proposition}
\label{prop_dependence}Let us consider two different systems, with identical
parameters except from $\left(  V^{\left(  1\right)  },V^{(2)}\right)  $
(first system) and $\left(  \tilde{V}^{\left(  1\right)  },\tilde{V}%
^{(2)}\right)  $ (second system). As previously, a tilde ($\sim$) is added to
any quantity referring to the second system. Assume that $\left(  V^{\left(
1\right)  },V^{(2)}\right)  \leq_{lo}\left(  \tilde{V}^{\left(  1\right)
},\tilde{V}^{(2)}\right)  $. Then $\tau$ is smaller than $\tilde{\tau}$ in the
sense of the standard stochastic order $(\tau\leq_{st}\tilde{\tau})$.
\end{proposition}

\begin{proof}
The point is to show that $\phi_{t}(L)\leq\tilde{\phi}_{t}(L)$. Starting from
$\left(  \ref{eq_phi2}\right)  $ and conditioning on $\sigma\left(  \left(
N_{t}\right)  _{t\geq0}\right)  $, we have
\[
\phi_{t}(L)=\mathbb{E}\left(  \Theta(N_{t},T_{1},\ldots,T_{N_{t}})\right)
\]
where
\begin{align*}
&  \Theta(n,t_{1},\ldots,t_{n})\\
&  =\mathbb{E}\left(  F_{G_{t}}\left(  L-\sum_{i=1}^{N_{t}}V_{i}^{(2)}\right)
\prod_{i=1}^{N_{t}}q(T_{i})e^{-\sum_{i=1}^{N_{t}}V_{i}^{(1)}(t-T_{i}%
)}\left\vert N_{t}=n,T_{1}=t_{1},\ldots,T_{n}=t_{n}\right.  \right) \\
&  =\mathbb{E}\left(  F_{G_{t}}\left(  L-\sum_{i=1}^{n}V_{i}^{(2)}\right)
\prod_{i=1}^{n}q(t_{i})e^{-\sum_{i=1}^{n}V_{i}^{(1)}(t-t_{i})}\right) \\
&  =\mathbb{E}\left(  k_{1}\left(  \sum_{i=1}^{n}V_{i}^{(2)}\right)
k_{2}\left(  \sum_{i=1}^{n}V_{i}^{(1)}(t-t_{i})\right)  \right)
\end{align*}
for all $n\in\mathbb{N}$, all $0\leq t_{1}\leq\ldots\leq t_{n}\leq t$, with
\begin{align*}
k_{1}\left(  v_{1}\right)   &  =F_{G_{t}}\left(  L-v_{1}\right)  ,\\
k_{2}\left(  v_{2}\right)   &  =\prod_{i=1}^{n}q(t_{i})e^{-v_{2}}.
\end{align*}
Based on the independence between $\left(  V^{\left(  1\right)  }%
,V^{(2)}\right)  $ and $\left(  N_{t}\right)  _{t\geq0}$, it is sufficient to
show that
\[
\Theta(n,t_{1},\ldots,t_{n})\leq\tilde{\Theta}(n,t_{1},\ldots,t_{n}),
\]
all $n\in\mathbb{N}$, all $0\leq t_{1}\leq\ldots\leq t_{n}\leq t$ to get
$\phi_{t}(L)\leq\tilde{\phi}_{t}(L)$. As $((t-t_{i})x_{1},x_{2})$ is non
decreasing in $x_{1}$ and $x_{2}$, we first know from \cite[Theorem
6.G.3.]{shaked2006stochastic} that
\[
\left(  (t-t_{i})V_{i}^{(1)},V_{i}^{(2)}\right)  \leq_{lo}\left(
(t-t_{i})\tilde{V}_{i}^{(1)},\tilde{V}_{i}^{(2)}\right)  ,
\]
all $1\leq i\leq n$. As $\left(  (t-t_{i})V_{i}^{(1)},V_{i}^{(2)}\right)
_{i=\overline{1,n}}$ and $\left(  (t-t_{i})\tilde{V}_{i}^{(1)},\tilde{V}%
_{i}^{(2)}\right)  _{i=\overline{1,n}}$ are two sequences of independent
random vectors, we derive from the same theorem that
\[
\left(  \sum_{i=1}^{n}V_{i}^{(1)}(t-T_{i}),\sum_{i=1}^{n}V_{i}^{(2)}\right)
\leq_{lo}\left(  \sum_{i=1}^{n}\tilde{V}_{i}^{(1)}(t-T_{i}),\sum_{i=1}%
^{n}\tilde{V}_{i}^{(2)}\right)  .
\]
As both functions $k_{1}$ and $k_{2}$ are non decreasing, we derive (same
theorem) that
\[
\Theta(n,t_{1},\ldots,t_{n})\leq\tilde{\Theta}(n,t_{1},\ldots,t_{n}),
\]
which achieves this proof.
\end{proof}

The previous result shows that the more dependent $V^{\left(  1\right)  }$ and
$V^{(2)}$ are, the larger the system lifetime is.

\subsection{Influence of the cumulated intensity function $\Lambda$ on the
lifetime $\tau$}

We finally study how the lifetime of the system depends on the frequency of shocks.

\begin{proposition}
\label{influence_lambda}Let us consider two different systems, with identical
parameters except from $\Lambda$ (first system) and $\tilde{\Lambda}$ (second
system). As previously, a tilde ($\sim$) is added to any quantity referring to
the second system. Assume that $\Lambda\geq\tilde{\Lambda}$, and that $q$ is
non decreasing. Then $\tau$ is smaller than $\tilde{\tau}$ in the sense of the
standard stochastic order $(\tau\leq_{st}\tilde{\tau})$.
\end{proposition}

\begin{proof}
Using a similar method as for the proof of Theorem \ref{propNBU}, we can write
$\phi_{t}(L)$ as
\[
\phi_{t}(L)=\mathbb{E}\left(  g_{\infty}\left(  T_{i}\right)  _{i=1}^{\infty
}\right)
\]
where
\[
g_{n}\left(  (t_{i})_{i=1}^{n}\right)  =\mathbb{E}\left(  F_{G_{t}}\left(
L-\sum_{i=1}^{n}V_{i}^{(2)}\mathbf{1}_{\left\{  t_{i}\leq t\right\}  }\right)
e^{\sum_{i=1}^{n}\left(  \ln\left(  q\left(  t_{i}\right)  \right)
-V_{i}^{(1)}(t-t_{i})\right)  \mathbf{1}_{\left\{  t_{i}\leq t\right\}  }%
}\right)
\]
is non decreasing with respect to each $t_{i}$, all $1\leq i\leq n$. As
$\Lambda\geq\tilde{\Lambda}$, we derive in the same way that
\[
\left(  T_{i}\right)  _{i=1}^{n}\leq_{sto}\left(  \tilde{T}_{i}\right)
_{i=1}^{n}%
\]
for all $n\geq1$, which allows to conclude.
\end{proof}

This result is very natural. The more frequent the shocks occur, the shorter
the lifetime is.

\section{Numerical experiments\label{section Num}}

\subsection{Validation of the results}

As already mentioned in Subsection \ref{rel}, a first possibility for
computing the system reliability $R_{L}(t)$ is to use classical Monte-Carlo
(MC) simulations (Method 1) and simulate a large number of independent
histories for the system up to time $t$. We here provide the algorithm that we
have used, considering the case where $h\left(  t\right)  =0$, all $t\geq0$
and where $V^{\left(  1\right)  }$ is almost surely positive ($\mathbb{P}%
\left(  V^{\left(  1\right)  }>0\right)  =1$) (not essential assumptions).

\begin{algorithm}
\label{algo}Repeat $M$ times with $M$ large enough:

\begin{enumerate}
\item Simulate $G_{t}$ with given distribution.

\item Simulate $N_{t}$ according to the Poisson distribution with parameter
$\Lambda(t)$.

\item Simulate $N_{t}$ i.i.d. random variables $W_{1},\ldots,W_{n}$ with
p.d.f. $\frac{\lambda(x)}{\Lambda(t)}1_{[0,t]}(x)$. The shock arrival times
are given by
\[
(T_{1},\ldots,T_{N_{t}})=(W_{(1)},\ldots,W_{(N_{t})}),
\]
where $(W_{(1)},\ldots,W_{(N_{t})})$ is the order statistics of $(W_{1}%
,\ldots,W_{N_{t}})$.

\item Simulate $Z_{i}$ as the result of a Bernoulli trial between a fatal (0)
and a non fatal shock (1) at time $T_{i}$, with probability $q(T_{i})$ for a
shock to be non fatal, all $i=1,\ldots,N_{t}$.

\item Simulate $N_{t}$ i.i.d. random vectors $\left(  V_{i}^{(1)},V_{i}%
^{(2)}\right)  $, $i=1,\ldots,N_{t}$ according to distribution $\mu$.

\item Simulate the lifetime $Y$ of the first component with conditional hazard
rate $A_{t}^{\left(  1\right)  }$ given $\mathcal{F}_{t}=\sigma\left(
A_{s},s\leq t\right)  $. With that aim, setting
\[
\kappa\left(  t\right)  =\int_{0}^{t}A_{s}^{\left(  1\right)  }ds=\sum
_{i=1}^{N_{t}}(t-T_{i})V_{i}^{(1)}%
\]
(see $\left(  \ref{int A1}\right)  $), $\kappa\left(  t\right)  $ is a
one-to-one increasing function from $[T_{1},+\infty)$ into $[0,+\infty)$ and,
for $\kappa\left(  T_{j}\right)  \leq u<\kappa\left(  T_{j+1}\right)  $ with
$j\geq1$, we have:%
\[
\kappa^{-1}\left(  u\right)  =\frac{u+\sum_{i=1}^{j}T_{i}V_{i}^{(1)}}%
{\sum_{i=1}^{j}V_{i}^{(1)}}.
\]
It is then known that, if $U$ is uniformly distributed on $\left[  0,1\right]
$, then $\kappa^{-1}\left(  -\ln\left(  U\right)  \right)  $ is identically
distributed as $Y$, see \cite[Proposition 1.20]{cocozza1998processus}.

\item Compute
\[
w^{\left(  j\right)  }=\mathbf{1}_{\{Y>t\}}\mathbf{1}_{\{G_{t}+\sum
_{i=1}^{N_{t}}V_{i}^{(2)}\leq L\}}\prod_{i=1}^{N_{t}}Z_{i}%
\]
where $j$ refers to the $j-$th MC history, with $1\leq j\leq M$.
\end{enumerate}
\end{algorithm}

At the end of the algorithm, symbol $w^{\left(  j\right)  }$ stands for the
realization of a Bernoulli trial $W$ between an up (1) or down (0) system at
time $t$, with probability $R_{L}\left(  t\right)  $ for the system to be up.
The reliability $R_{L}\left(  t\right)  $ is then classically approximated by
the empirical mean $m_{W}$ of the $w^{\left(  j\right)  }$'s and a 95\%
asymptotic confidence interval is computed.

Method 2 is based on MC simulations of trajectories of $\left(  A_{t}\right)
_{t\geq0}$ and of $\left(  N_{t}\right)  _{t\geq0}$ (see Section \ref{rel}).
For both methods 1 \& 2, MC simulations are based on $N=10^{5}$ histories.
Methods 3 \& 4 are described in Section \ref{rel}. The four methods are
compared on a few specific examples. In all these examples, we compare the
reliability at time $t=1$ ($R(1)$) and we suppose that the shock are due to a
homogeneous Poisson process with parameter $\lambda=1$ , $L=2$, $h=0$ and
$(G_{t})_{t\geq0}$ is a null process. All other parameters are provided in
Table \ref{table4}, where $T\hookrightarrow\mathcal{E}(1)$\ means that the
random variable $T$ is exponentially distributed with mean 1.
\begin{table}[ptb]
\caption{Validation of the results}%
\label{table4}
\begin{center}%
\begin{tabular}
[c]{ccccc}\hline
Input & $\left(  V^{(1)},V^{(2)}\right)  $ & Method & $R(1)$ & 95 \%
CI\\\hline
$q(x)=e^{-x}$ &  & 1 & 0.5196 & [0.5181 0.5212]\\
$V^{(1)}\hookrightarrow\mathcal{E}(1)$, $V^{(2)}\hookrightarrow\mathcal{E}(1)
$ & Independence & 2 & 0.5195 & [0.5184 0.5205]\\
&  & 3, 4 & 0.5198 & \\\hline
$q(x)=0.5$ &  & 1 & 0.5049 & [0.5033 0.5064]\\
$V^{(1)}=V^{(2)}\hookrightarrow\mathcal{E}(1)$ & Complete dependence & 2 &
0.5049 & [0.5039 0.5059]\\
&  & 4 & 0.5054 & \\\hline
$q(x)=e^{-x}$,$V^{(2)}=V^{(1)}+W^{(2)}$ &  & 1 & 0.4809 & [ 0.4793 0.4824]\\
$V^{(1)}$, $W^{(2)}$ independent & Dependence & 2 & 0.4813 & [ 0.4801
0.4825]\\
$V^{(1)}\hookrightarrow\mathcal{E}(1)$, $W^{(2)}\hookrightarrow\mathcal{E}(1)
$ &  &  &  & \\\hline
\end{tabular}
\vspace*{0.5cm}
\end{center}
\end{table}

Methods 1 and 2 may be used in any case. Method 2 is more effective than
Method 1 (shorter c.p.u. time and tighter 95\% confidence interval - IC -).
Method 3 is more practical when $V^{(1)}$ and $V^{(2)}$ are independent with
some specific distribution. Method 4 is adapted to the case where $V^{(1)}$
and $V^{(2)}$ are dependent.

\subsection{Examples}

We here illustrate several properties from a numerical point of view on a few
examples. Examples parameters are provided in Table \ref{table3}.

\begin{table}[ptb]
\caption{Parameters for the examples}%
\label{table3}
\begin{center}%
\begin{tabular}
[c]{ccccccccc}\hline
& $L$ & $h$ & $G_{t}$ & $q(t)$ & $\lambda$ & $V^{(1)}$ & $V^{(2)}$ & $\left(
V^{(1)},V^{(2)}\right)  $\\\hline
Ex.\ref{ex9} & - & 0 & 0 & - & 1 & $\mathcal{E}(1)$ & $\mathcal{E}(1)$ &
Independent\\
Ex.\ref{ex8} & 2 & 0 & 0 & 1 & 1 & $\mathcal{E}(1)$ & $\mathcal{E}(1)$ & -\\
Ex.\ref{ex6} & 2 & 0 & 0 & $1-e^{-x}$ & - & 1 & $\mathcal{E}(1)$ &
Independent\\\hline
\end{tabular}
\vspace*{0.5cm}
\end{center}
\end{table}

\begin{example}
\label{ex9} This example illustrates the NBU property of the lifetime when
$\lambda$ is constant and $q(x)=e^{-x}$ is non increasing, see Theorem
\ref{propNBU}. Taking $L=2$, Fig. \ref{fig7_1} indeed shows that the remaining
lifetime of a system with age $t_{0}=1$ is stochastically smaller than the
lifetime of a new system. On the contrary, when $\lambda$ is still constant
but $q(x)=1-e^{-3x}$ is non decreasing, the remaining lifetime of a system
with age $t_{0}=1$ is not comparable with that of a new system, see Fig.
\ref{fig7_2} with $L=4$. So the NBU property does not hold anymore in that case.

\begin{figure}[ptb]
\centering
\includegraphics[scale=0.32,angle=0]{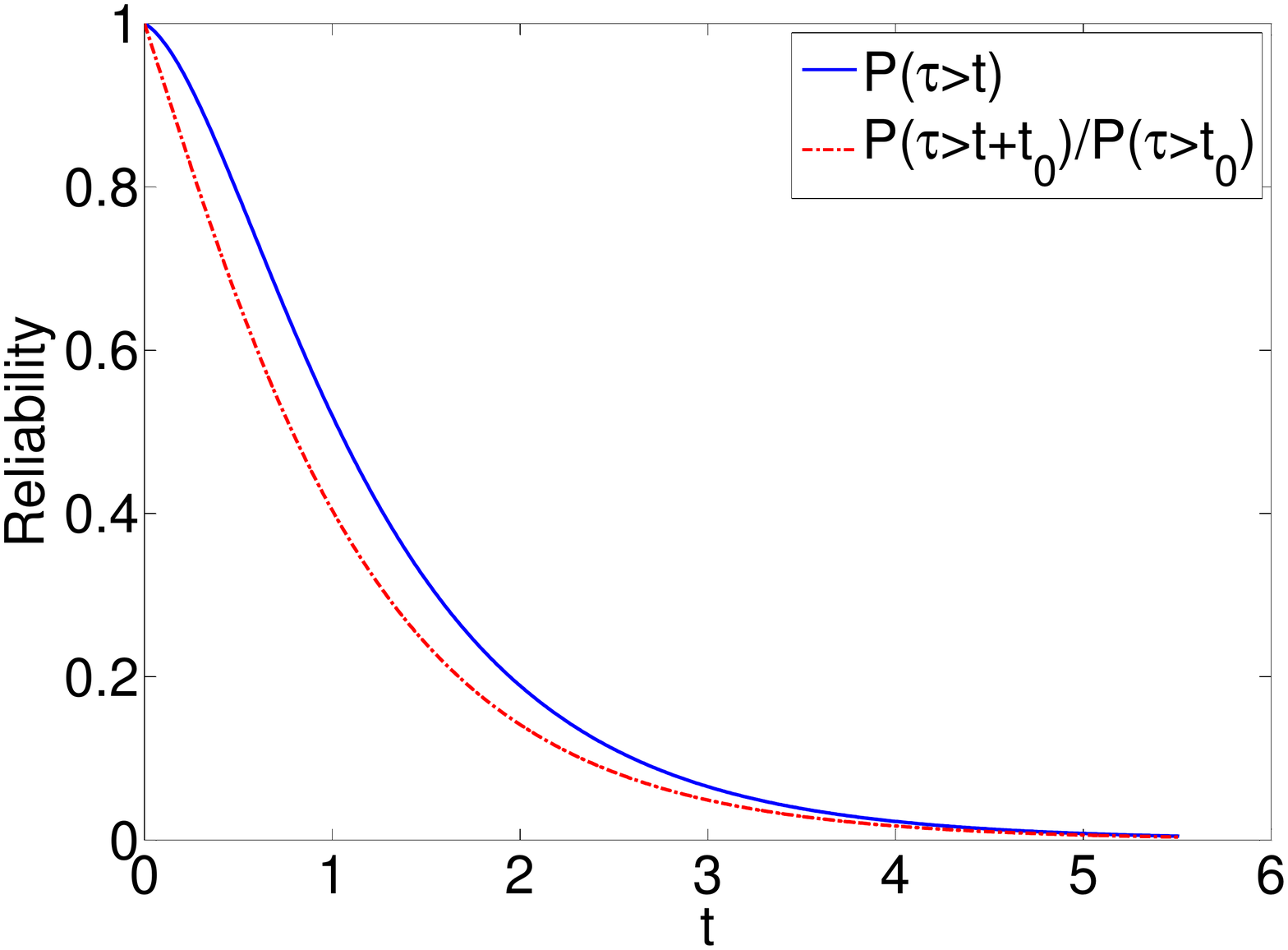} \vspace*{-0.5cm}%
\caption{Example \ref{ex9}, NBU case}%
\label{fig7_1}%
\end{figure}

\begin{figure}[ptb]
\centering\includegraphics[scale=0.32,angle=0]{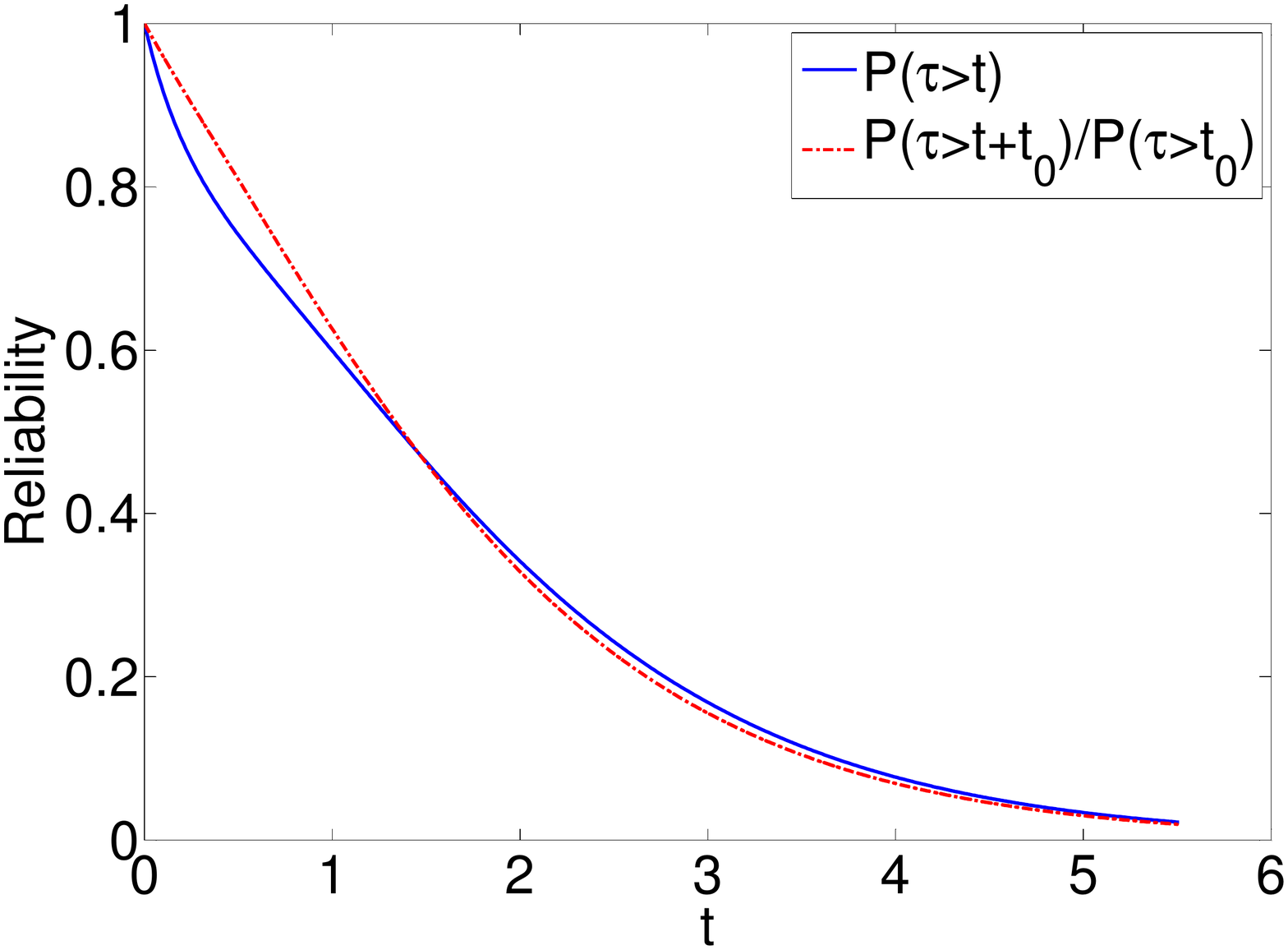}\vspace*{-0.5cm}%
\caption{Example \ref{ex9}, not NBU case}%
\label{fig7_2}%
\end{figure}
\end{example}

\begin{example}
\label{ex8}We here consider two extreme cases for the dependency between
$V^{\left(  1\right)  }$ and $V^{(2)}$: independent or completely dependent
(here $V^{\left(  1\right)  }=V^{(2)}$). The reliability in the completely
dependent case is always greater than in the independent case (Fig.
\ref{fig6}). This result is coherent with Proposition \ref{prop_dependence}.

\begin{figure}[ptb]
\centering\includegraphics[scale=0.32,angle=0]{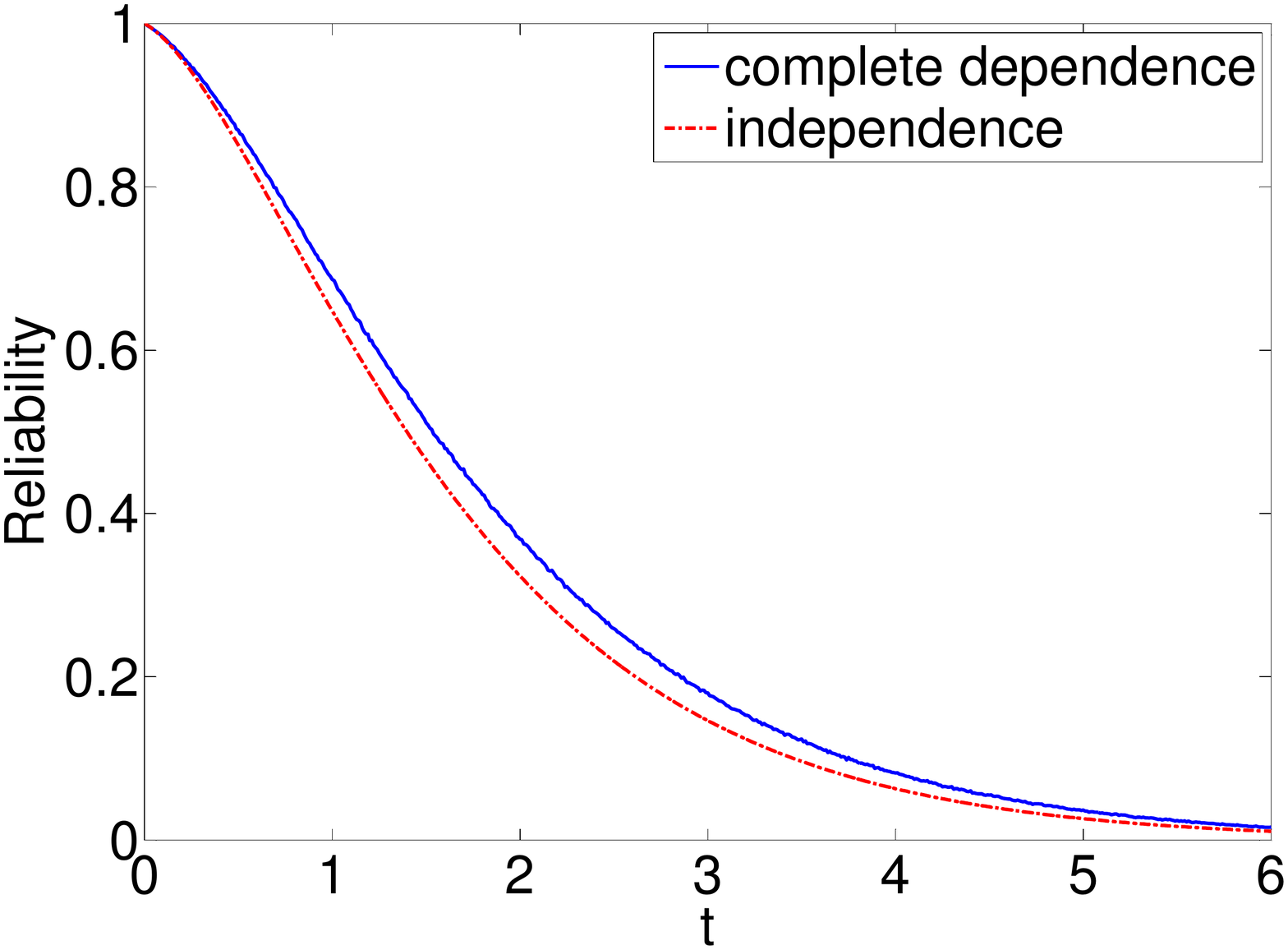}\vspace*{-0.5cm}%
\caption{Comparison of reliability for two different types of dependence
between $V^{(1)}$ and $V^{(2)}$, Example \ref{ex8}}%
\label{fig6}%
\end{figure}
\end{example}

\begin{example}
\label{ex6}This example shows the monotony of the reliability with respect to
the intensity of the Poisson process $\lambda\left(  x\right)  $ when $q$ is
increasing, see Proposition \ref{influence_lambda}. The more frequently the
shocks occur, the lower the reliability is (Fig. \ref{fig3}).

\begin{figure}[ptb]
\centering\includegraphics[scale=0.32,angle=0]{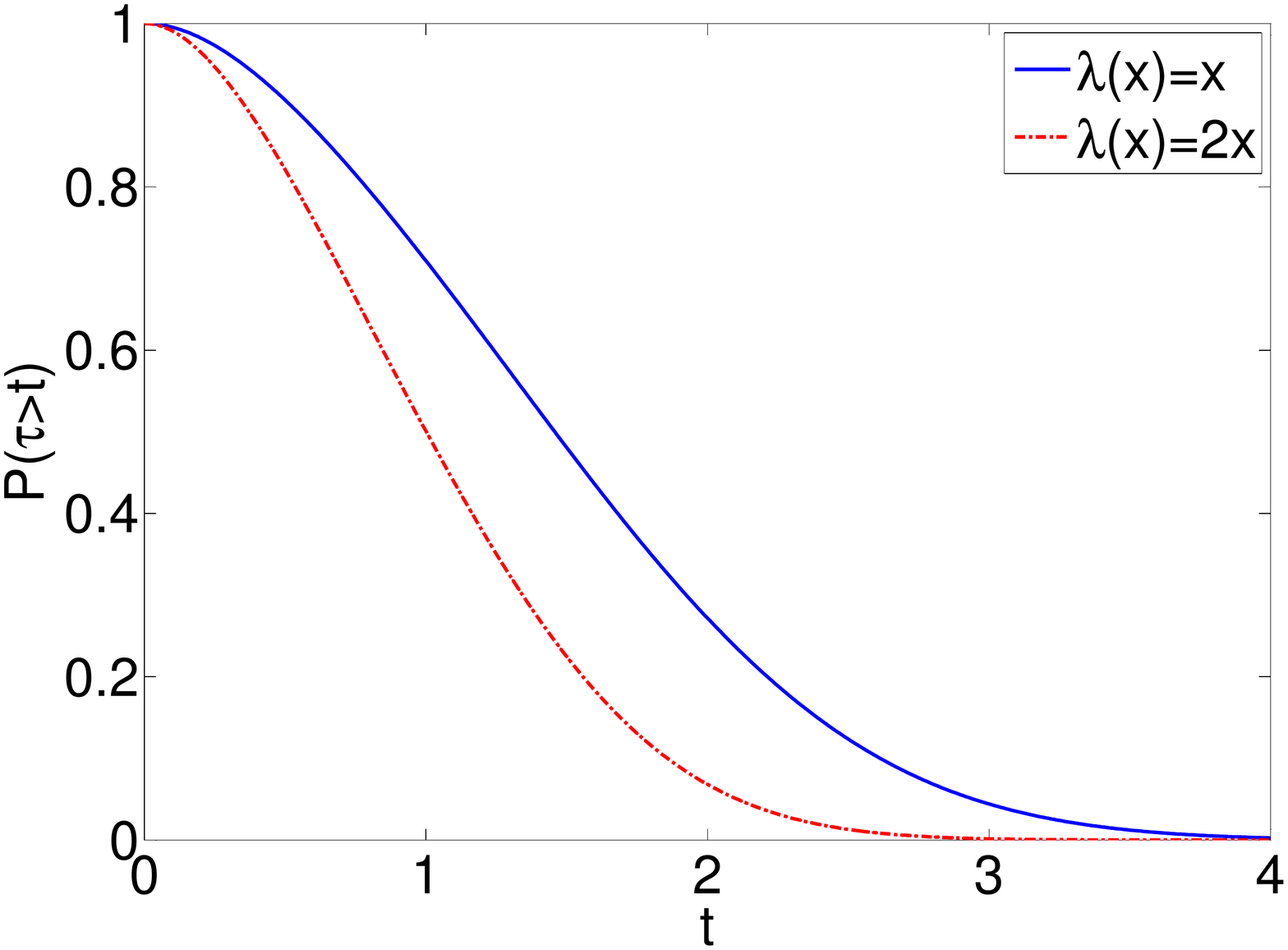}\vspace*{-0.5cm}%
\caption{Comparison of reliability for two different $\lambda\left(  x\right)
$, Example \ref{ex6}}%
\label{fig3}%
\end{figure}
\end{example}

\section{Concluding remarks\label{concl}}

We here proposed a random shock model with competing failure modes, which
enlarges several models from the previous literature. The model takes into
account different types of dependence between competing failures modes, where
the dependence is induced by a common external shock environment. The
reliability has been calculated by several different methods and conditions
have been provided under which the system lifetime is New Better than Used.
Due to this ageing property, it might be of interest to propose and study some
maintenance policy to enlarge the system lifetime. Several versions might be
proposed, according to the available information.

Also, the influence of the characteristics of the stressing environment on the
lifetime $\tau$ has been studied. As expected, we saw that the lifetime was
stochastically increasing with the probability $q\left(  \cdot\right)  $ for a
shock to be non fatal. Besides, and that result was not necessarily so clear
at first sight, we saw that the lifetime was also stochastically increasing
with the dependence between the two marginal shock sizes. Finally, in case of
a non decreasing function $q\left(  \cdot\right)  $, we saw that the lifetime
was stochastically decreasing with the cumulated frequency of shocks. This
means that the more frequent the shocks occur, the shorter the lifetime is.
This result is natural but our proof is limited to the case of a non
decreasing function $q\left(  \cdot\right)  $. In the special case of
\cite{cha2011stochastic}, the survival function of $\tau$ is however given by
\[
\mathbb{P}(\tau>t)=e^{-H(t)-\int_{0}^{t}\left(  1-\tilde{\mu}_{1}%
(t-w)q(w)\right)  \lambda(w)~dw}%
\]
and it is easy to check that if $\lambda\geq\tilde{\lambda}$ (stronger
assumption than $\Lambda\geq\tilde{\Lambda}$) then $\tau\leq_{st}\tilde{\tau}%
$, without any special condition on $q\left(  \cdot\right)  $. So, the
stochastic monotonicity of the lifetime with respect of the (cumulated ?)
frequency of shocks might be true under a more general setting than in the
present paper, without assuming any monotonicity condition on $q\left(
\cdot\right)  $. We however have not been able to conclude on this point, and
whether it is true or not remains an open question.

\begin{acknowledgement}
Both authors thank the referees for their carefull reading of the paper and their constructive remarks, which lead to a better introduction and justification of the model, and to a clearer paper.
This work has been initiated during Hai Ha PHAM's PhD studies in Pau (France), and has been supported by the Conseil R\'{e}gional d'Aquitaine (France). 
This work has also been supported for both authors by the French National Research Agency (AMMSI project, ref. ANR 2011 BS01-021).
\end{acknowledgement}


\begin{thebibliography}{999999999999999999999999999999999999999999999999999999999}                                        %
\bibitem[A-Hameed and Proschan(1973)]{abdul1972nonstationary}M.~S. A-Hameed
and F.~Proschan. \newblock Nonstationary shock models.
\newblock \emph{Stochastic Processes and their Applications}, \textbf{1}%
\penalty0 (4):\penalty0 383--404, 1973.

\bibitem[A-Hameed and Proschan(1975)]{proschan1975shock}M.~S. A-Hameed and
F.~Proschan. \newblock Shock models with underlying birth process.
\newblock
\emph{Journal of Applied Probability}, \textbf{12}\penalty0 (1):\penalty0
18--28, 1975.

\bibitem[Brown and Proschan(1983)]{brown1983imperfect}M.~Brown and
F.~Proschan. \newblock Imperfect repair. \newblock \emph{Journal of Applied
Probability}, \textbf{20}\penalty0 (4):\penalty0 851--859, 1983.

\bibitem[Cha and Finkelstein(2009)]{cha2009terminating}J.~H. Cha and
M.~Finkelstein. \newblock On a terminating shock process with independent wear
increments. \newblock \emph{Journal of Applied Probability}, \textbf{46}%
\penalty0 (2):\penalty0 353--362, 2009.

\bibitem[Cha and Mi(2007)]{cha2007study}J.~H. Cha and J.~Mi. \newblock Study
of a stochastic failure model in a random environment. \newblock \emph{Journal
of Applied Probability}, \textbf{44}\penalty0 (1):\penalty0 151--163, 2007.

\bibitem[Cha and Mi(2011)]{cha2011stochastic}J.~H. Cha and J.~Mi. \newblock On
a stochastic survival model for a system under randomly variable environment.
\newblock \emph{Methodology and Computing in Applied Probability},
\textbf{13}\penalty0 (3):\penalty0 549--561, 2011.

\bibitem[{\c{C}}inlar(2011)]{ccnlar2011probability}Erhan {\c{C}}inlar.
\newblock \emph{Probability and stochastics}, volume 261 of \emph{Graduate
texts in Mathematics}. \newblock Springer Science + Business Media, 2011.

\bibitem[Cocozza-Thivent(1998)]{cocozza1998processus}C.~Cocozza-Thivent.
\newblock \emph{Processus stochastiques et fiabilit{\'e} des syst{\`e}mes},
volume~28 of \emph{Math\'{e}matiques et Applications}. \newblock Springer, 1998.

\bibitem[Esary et~al.(1973)Esary, Marshall, and Proschan]{esary1973shock}J.~D.
Esary, A.~W. Marshall, and F.~Proschan. \newblock Shock models and wear
processes. \newblock \emph{The Annals of Probability}, \textbf{1}\penalty0
(4):\penalty0 627--649, 1973.

\bibitem[\color{black}Finkelstein and Cha(2013)]{finkelstein2013stochastic}%
\color{black}M.~Finkelstein and J.~H. Cha.
\newblock \emph{\color{black}Stochastic modeling for reliability: Shocks,
burn-in and heterogeneous populations}. \newblock \color{black}Springer Series
in Reliability Engineering. Springer, London, 2013.\color{black}

\bibitem[Gut(2001)]{gut2001mixed}A.~Gut. \newblock Mixed shock models.
\newblock \emph{Bernoulli}, \textbf{7}\penalty0 (3):\penalty0 541--555, 2001.

\bibitem[Gut and H{\"u}sler(1999)]{gut1999extreme}A.~Gut and J.~H{\"u}sler.
\newblock Extreme shock models. \newblock \emph{Extremes}, \textbf{2}\penalty0
(3):\penalty0 295--307, 1999.

\bibitem[\color{black}Hao et~al.(2013)Hao, Su, and Qu]{Hao2013}%
\color{black}H.-B. Hao, C.~Su, and Z.-Z. Qu.
\newblock \color{black}Reliability analysis for mechanical components subject
to degradation process and random shock with wiener process.
\newblock \color{black}In \emph{\color{black}19th International Conference on
Industrial Engineering and Engineering Management}\color{black}, pages
531--543, 2013.\color{black}

\bibitem[\color{black}Lehmann(2006)]{lehmann2006}\color{black}A.~Lehmann.
\newblock \color{black}Degradation-threshold-shock models.
\newblock \emph{\color{black}Probability, Statistics and Modelling in Public
Health}\color{black}, pages 286--298, 2006.\color{black}

\bibitem[\color{black}Lehmann(2009)]{lehmann2009}\color{black}A.~Lehmann.
\newblock \color{black}Joint modeling of degradation and failure time data.
\newblock \emph{\color{black}Journal of Statistical Planning and Inference},
\color{black}139\penalty0 (5):\penalty0 1693 -- 1706, 2009.
\newblock \color{black}Special Issue on Degradation, Damage, Fatigue and
Accelerated Life Models in Reliability Testing.\color{black}

\bibitem[\color{black}Lemoine and Wenocur(1985)]{lemoine1985}%
\color{black}A.~J. Lemoine and M.~L. Wenocur. \newblock \color{black}On
failure modeling. \newblock \emph{\color{black}Naval Research Logistics
Quarterly}\color{black}, 32\penalty0 (3):\penalty0 497--508, 1985.\color{black}

\bibitem[Mallor and Omey(2001)]{mallor2001shocks}F.~Mallor and E.~Omey.
\newblock Shocks, runs and random sums. \newblock \emph{Journal of Applied
Probability}, \textbf{38}\penalty0 (2):\penalty0 438--448, 2001.

\bibitem[Mallor and Santos(2003{\natexlab{a}})]{mallor2003classification}%
F.~Mallor and J.~Santos. \newblock Classification of shock models in system
reliability. \newblock \emph{Monograf{\'i}as del Semin. Matem. Garc{\'i}a de
Galdeano}, \textbf{27}:\penalty0 405--412, 2003{\natexlab{a}}.

\bibitem[Mallor and Santos(2003{\natexlab{b}})]{mallor2003reliability}%
F.~Mallor and J.~Santos. \newblock Reliability of systems subject to shocks
with a stochastic dependence for the damages. \newblock \emph{Test},
\textbf{12}\penalty0 (2):\penalty0 427--444, 2003{\natexlab{b}}.

\bibitem[Marshall and Shaked(1979)]{marshall1979multivariate}A.W. Marshall and
M.~Shaked. \newblock Multivariate shock models for distributions with
increasing hazard rate average. \newblock \emph{The Annals of Probability},
\textbf{7}\penalty0 (2):\penalty0 343--358, 1979.

\bibitem[\color{black}Nakagawa(2007)]{nakagawa2007shock}%
\color{black}T.~Nakagawa. \newblock \emph{\color{black}Shock and damage models
in reliability theory}. \newblock \color{black}Springer Series in Reliability
Engineering. Springer, London, 2007.\color{black}

\bibitem[Qian et~al.(1999)Qian, Nakamura, and Nakagawa]{qian1999cumulative}%
C.~Qian, S.~Nakamura, and T.~Nakagawa. \newblock Cumulative damage model with
two kinds of shocks and its application to the backup policy. \newblock
\emph{Journal of the Operations Research Society of Japan-Keiei Kagaku},
\textbf{42}\penalty0 (4):\penalty0 501--511, 1999.

\bibitem[Savits(1988)]{savits1988some}T.~H. Savits. \newblock Some
multivariate distributions derived from a non-fatal shock model. \newblock
\emph{Journal of Applied Probability}, \textbf{25}\penalty0 (2):\penalty0
383--390, 1988.

\bibitem[Shaked and Shanthikumar(2006)]{shaked2006stochastic}M.~Shaked and
J.~G. Shanthikumar. \newblock \emph{Stochastic orders}. \newblock Springer
Series in Statistics. Springer, 2006.

\bibitem[Singpurwalla(1995)]{singpurwalla1995survival}N.~D. Singpurwalla.
\newblock Survival in dynamic environments. \newblock \emph{Statistical
Science}, \textbf{10}\penalty0 (1):\penalty0 86--103, 1995.

\bibitem[Skoulakis(2000)]{skoulakis2000}G.~Skoulakis. \newblock A general
shock model for a reliability system. \newblock \emph{Journal of Applied
Probability}, \textbf{37}\penalty0 (4):\penalty0 925--935, 2000.

\bibitem[\color{black}Wang and Gao(2014)]{Wang2014}\color{black}H.W. Wang and
J.~Gao. \newblock \color{black}A reliability evaluation study based on
competing failures for aircraft engines.
\newblock \emph{\color{black}Eksploatacja i Niezawodnosc: Maintenance and
Reliability}, 16\penalty0 (2):\penalty0 171--178, 2014.\color{black}

\bibitem[\color{black}Zhu et~al.(2010)Zhu, Elsayed, Liao, and Chan]%
{Zhu2010}\color{black}Y.~Zhu, E.A. Elsayed, H.~Liao, and L.Y. Chan.
\newblock \color{black}Availability optimization of systems subject to
competing risk. \newblock \emph{\color{black}European Journal of Operational
Research}, \color{black}202\penalty0 (3):\penalty0 781 -- 788, 2010.\color{black}
\end{thebibliography}
\end{document}